\documentclass[sn-mathphys-num]{sn-jnl}

\usepackage{graphicx}%
\usepackage{multirow}%
\usepackage{amsmath,amssymb,amsfonts}%
\usepackage{amsthm}%
\usepackage{mathrsfs}%
\usepackage[title]{appendix}%
\usepackage{xcolor}%
\usepackage{textcomp}%
\usepackage{manyfoot}%
\usepackage{booktabs}%
\usepackage{algorithm}%
\usepackage{algorithmicx}%
\usepackage{algpseudocode}%
\usepackage{listings}%

\usepackage{enumerate}

\usepackage[all, pdf, 2cell]{xy} 

\theoremstyle{thmstyleone}%
\newtheorem{theorem}{Theorem}[section]
\newtheorem{lemma}[theorem]{Lemma}
\newtheorem{proposition}[theorem]{Proposition}
\newtheorem{corollary}[theorem]{Corollary}

\theoremstyle{thmstyletwo}%
\newtheorem{remark}[theorem]{Remark}
\newtheorem{example}[theorem]{Example}

\theoremstyle{thmstylethree}%
\newtheorem{definition}[theorem]{Definition}

\usepackage{mathdots}

\newcommand*{\defeq}{\mathrel{\vcenter{\baselineskip0.5ex \lineskiplimit0pt
                     \hbox{\scriptsize.}\hbox{\scriptsize.}}}%
                     =}

\begin{document}

\author*[1]{\fnm{Haoxu} \sur{Wang}}\email{whx17@mails.tsinghua.edu.cn}

\affil[1]{\orgdiv{Yau Mathematical Sciences Center}, \orgname{Tsinghua University}, \orgaddress{\city{Beijing}, \postcode{100084}, \country{China}}}

\title{Newton Polyhedrons and Hodge Numbers of Non-degenerate Laurent Polynomials}

\abstract{Claude Sabbah has defined the Fourier transform $G$ of the Gauss-Manin system for a non-degenerate and convenient Laurent polynomial and has shown that there exists a polarized mixed Hodge structure on the vanishing cycle of $G$.
In this article, we consider certain non-degenerate and convenient Laurent polynomials $f_{P,\mathbf{a}}$, whose Newton polyhedron at infinity is a simplicial polytope $P$.
First, we consider the stacky fan $\boldsymbol{\Sigma}_P$ given by $P$ and show that for each quotient stacky fan of $\boldsymbol{\Sigma}_P$, there is a natural polarized mixed Hodge structure on the ring of conewise polynomial functions on it.
Then, we describe the polarized mixed Hodge structure on the vanishing cycle associated to $f_{P,\mathbf{a}}$ using these rings of conewise polynomial functions.
In particular, we compute the Hodge diamond of the vanishing cycle. 
As a further consequence, we can solve the Birkhoff problem of such a Laurent polynomial by using elementary methods.}

\keywords{Brieskorn lattice, vanishing cycle, polarized mixed Hodge structure, polytope, stacky fan, conewise polynomial function}

\pacs[MSC Classification]{14F40, 14M25, 32S35, 32S40, 52B20}

\maketitle

\tableofcontents

\section{Introduction}

Let $K$ be a field of characteristic $0$,
let $N \cong \mathbb{Z}^n$ be a free abelian group of rank $n < +\infty$.
Consider $K\left[\mathbf{t}^{ \pm 1}\right] = K[t_1^{ \pm 1}, \cdots, t_n^{ \pm 1}]$, where $t_1, \cdots, t_n$ is a basis of $\operatorname{hom}_{\mathbb{Z}}(N, \mathbb{Z})$.
Let $U \defeq \operatorname{Spec} K[\mathbf{t}^{\pm 1}]$.
Consider a Laurent polynomial 
$$
f
=\sum_{j=1}^{N} a_{j} t_{1}^{w_{1 j}} \cdots t_{n}^{w_{n j}} 
=\sum_{j=1}^{N} a_{j} \mathbf{t}^{w_j}
\in K[\mathbf{t}^{\pm 1}].
$$
Let $P$ be the \emph{Newton polyhedron} of $f$ at $\infty$, that is, the convex hull of the set 
$\left\{0, w_1, \ldots, w_N\right\}$ 
in $\mathbb{Q}^n$, where $w_j = (w_{1 j}, \cdots, w_{n j}) \in \mathbb{Q}^n$.
For any face $F$ of $P$, denote 
$f_F = \sum_{w_j \in F} a_{j} \mathbf{t}^{w_j}$. 
We say $f$ is non-degenerate if for any face $F$ of $P$ not containing $0$, the equations
$$
\frac{\partial f_F}{\partial t_1}=\cdots=\frac{\partial f_F}{\partial t_n}=0
$$
define an empty sub-scheme in $U$. 
We say $f$ is convenient if $0$ lies in the interior of $P$.
If $f$ is non-degenerate and convenient,
we may construct the Fourier transform $G$ of the Gauss-Manin system associated to $f$, its Brieskorn lattice $G_0$,  and its vanishing cycle $H$.
See Section \ref{section:Brieskorn} or \cite{sabbah1999hypergeometric, douai2003gauss}.
By \cite{steenbrink1985variation, saito1989structure, sabbah1997monodromy}, we know that $H$ has a polarized mixed Hodge structure $(H_{\mathbb{Q}}, F^{\bullet}, N, Q)$. 

Let $P$ be a polytope in $N_{\mathbb{Q}} \defeq N \otimes_{\mathbb{Z}} \mathbb{Q}$.
Denote by $P(0)$ the set of vertex of $P$.
Assume that 
\begin{enumerate}[(a)]
\item $P$ is a lattice polytope with respect to $N$, i.e. $P(0) \subset N$,

\item $P$ is a simplicial polytope, i.e. each facet of $P$ contains exactly $n$ vertices,

\item $0$ lies in the interior of $P$.
\end{enumerate}
Consider a Laurent polynomial of the form
$$\begin{aligned}
f = f_{P, \mathbf{a}} \defeq \sum_{v \in P(0)} a_v \mathbf{t}^v \in K[\mathbf{t}^{\pm 1}],
\end{aligned}$$
where $a_v \in K^*$ for all $v \in P(0)$.
We can show that $f_{P, \mathbf{a}}$ is non-degenerate (Lemma \ref{lemma:f_conv_non_dege}).
The aim of this paper is to describe the mixed Hodge structure $(H, F^{\bullet}, N)$ associated to $f_{P, \mathbf{a}}$ using the combinatorial data of $P$ (Corollary \ref{cor:Birkhoff problem}(1)).
In particular, we calculate the Hodge number of $H$ (Corollary \ref{cor:PMHS-van-cyc}).

Many articles (e.g. \cite{tanabe2004combinatorial, harder2021hodge}, etc.) attempt to compute these Hodge numbers (or the spectra) associated to non-degenerate Laurent polynomials.
In particular, in \cite{douai2018global, douai2021hard}, for a lattice simplicial polytope $P$,
Douai shows that for the Laurent polynomial
$$
f_{P, \mathbf{1}} = \sum_{v \in P(0)} \mathbf{t}^v,
$$
we can use the combinatorial data of $P$ to describe the Jacobian ring $J_{f_{P, \mathbf{1}}}$.
More precisely, by \cite{borisov2005orbifold}, we can use $P$ to construct a toric Deligne-Mumford stack $\mathcal{X}(\boldsymbol{\Sigma}_P)$ and the orbifold Chow ring of $\mathcal{X}(\boldsymbol{\Sigma}_P)$ is isomorphic to the graded Jacobian ring $\operatorname{Gr}^{\mathcal{N}} J_{f_{P, \mathbf{1}}}$, where $\mathcal{N}$ is the Newton filtration on $J_{f_{P, \mathbf{1}}}$.
We can decompose the orbifold Chow ring to a direct sum of the Stanley-Reisner ring of some fans which can be easily described by $P$.

In \cite{sabbah2018some}, Sabbah considered the Laurent polynomial $f_{P, \mathbf{1}}$ where $P$ is a smooth Fano polytope.
He shows that we can relate the vanishing cycle $H$ to the Jacobian ring $\operatorname{Gr}^{\mathcal{N}} J_{f_{P, \mathbf{1}}}$ and therefore to the Chow ring of the toric variety defined by $P$.
By this way, he shows that the mixed Hodge structure on $H$ is of Hodge-Tate type, i.e. the Hodge number $h^{p,q} = 0$ for all $p \neq q$.
Using deformation methods, he then shows that this result holds for $f_{P, \mathbf{a}}$ for all $\mathbf{a} \in (K^*)^{P(0)}$.

In this paper, without resorting to toric geometries and deformation methods, we will show that for any lattice simplicial polytope $P$ and for any $\mathbf{a} \in (K^*)^{P(0)}$, we can decompose the vanishing cycle and the graded Jacobian ring to a direct sum of some sub-spaces respectively, see (\ref{align:decomposition-H}) and (\ref{align:decomposition-jac}). 
Each of them is isomorphic to the Stanley-Reisner ring of a fan, see Corollary \ref{lemma:van-jac-rel-u}.
As another corollary, we can solve the Birkhoff problem by elementary methods (See Corollary \ref{cor:Birkhoff problem} (2)).

\section{Polytopes and stacky fans}
\label{section:Polytopes}

\subsection{Stacky fans}

Let $N$ be a finitely generated abelian group.
We will consider polytopes, cones and fans etc. in $N_{\mathbb{Q}}$.
Denote by $\overline{N}$ the image of $N$ in $N_{\mathbb{Q}}$.
Hence $\overline{N} \cong \mathbb{Z}^n$ for some $n$.
Similarly, for any $u \in N$, denote by $\overline{u}$ the image of $u$ in $N_{\mathbb{Q}}$.
Unless otherwise stated, we assume that $N \xrightarrow{\sim} \overline{N}$.

We denote by $P(k)$ the set of all $k$-dimensional faces of a polyhedra $P$, i.e. an intersection of a finite number of affine half spaces in $N_{\mathbb{Q}}$.
Similarly, we denote by $\Sigma(k)$ the set of all $k$-dimensional cones in a fan $\Sigma$.

\begin{definition}[\cite{borisov2005orbifold}]
\label{def:stacky_fan}
A stacky fan $\boldsymbol{\Sigma} = (N, \Sigma, \{v_{\rho}\}_{\rho \in \Sigma(1)})$ is a triple consisting of 
a finitely generated abelian group $N$, 
a simplicial fan $\Sigma$ in $\mathbb{Q} \otimes_{\mathbb{Z}} N$,
and $v_{\rho} \in N$ for each ray $\rho \in \Sigma(1)$ such that $\overline{v}_{\rho}$ is a generator of $\rho$.
\end{definition}

\begin{example}
\label{exm:stacky_fan}
\begin{enumerate}[(i)]
\item 
A simplicial fan $\Sigma$ in $N_{\mathbb{Q}}$ determines a stacky fan $(\overline{N}, \Sigma, \{v_{\rho}\}_{\rho \in \Sigma(1)})$ where $v_{\rho}$ is the minimal lattice points on the rays. 

\item
\label{item:exm_stacky_fan_poly}
Let $P$ be a lattice simplicial polytope containing the origin as an interior point.
Then $P$ determines a stacky fan $\boldsymbol{\Sigma}_P = (N, \Sigma_P, P(0))$, where the cones in $\Sigma_P$ are the cones over proper faces of $P$.
\end{enumerate}
\end{example}

Let $\boldsymbol{\Sigma}$ be a stacky fan.
Notice that
\begin{align}\label{align:decomp-cone}
|\Sigma| = \bigcup_{\sigma \in \Sigma} \sigma = \bigsqcup_{\sigma \in \Sigma} 
\sigma^{\circ}
\end{align}
where for each $\sigma \in \Sigma$,
$$
\sigma^{\circ} \defeq \left\{\sum_{\rho \in \sigma(1)} \lambda_{\rho} \overline{v}_{\rho} \middle|  \lambda_{\rho} > 0\right\}
$$
is the relative interior of $\sigma$.
For any $u \in |\Sigma|$, denote by $\sigma(u)$ the unique cone in $\Sigma$ such that $u \in \sigma(u)^{\circ}$.

For each cone $\sigma \in \Sigma$, denote
$$\begin{aligned}
\operatorname{Box}(\sigma) &\defeq \left\{u \in N \middle| \overline{u}=\sum_{\rho \in \sigma(1)} \lambda_{\rho} \overline{v}_{\rho} \text{ for some } 0 \leq \lambda_{\rho}<1\right\},\\
P(\sigma) &\defeq \left\{u \in N \middle| u = \sum_{\rho \in \sigma(1)} \lambda_{\rho} v_{\rho} \text{ for some } \lambda_{\rho} \in \mathbb{Z}_{\geq 0}\right\}.
\end{aligned}$$
Denote 
\begin{align}\label{align:def-box-P}
\operatorname{Box}(\boldsymbol{\Sigma}) \defeq \bigcup_{\sigma \in \Sigma} \operatorname{Box}(\sigma) \quad \text{ and } \quad P(\boldsymbol{\Sigma}) \defeq \bigcup_{\sigma \in \Sigma} P(\sigma). 
\end{align}
Then for any $u \in N$, there exists a unique element $\{u\} \in \operatorname{Box}(\boldsymbol{\Sigma})$, and a unique element $\lfloor u \rfloor \in P(\boldsymbol{\Sigma})$, such that $u = \{u\} + \lfloor u \rfloor$.

For any $u \in \operatorname{Box}(\boldsymbol{\Sigma})$, denote
\begin{align}\label{align:def-P_u}
P_u(\boldsymbol{\Sigma}) 
\defeq \left\{w \in N \middle| \{w\} = u\right\}.
\end{align}
Then 
\begin{align}\label{align:decomp-P_u}
N = \bigsqcup_{u \in \operatorname{Box}(\boldsymbol{\Sigma})} P_u(\boldsymbol{\Sigma}).
\end{align}

\subsection{Conewise polynomial functions}

\begin{definition}[\cite{barthel2002combinatorial, braden2006remarks, fleming2010hard}]
\label{def:conewise_polynomial}
Let $K$ be a field of characteristic $0$.
Let $K\left[\mathbf{t}\right] = K[t_1, \cdots, t_n]$ be the ring of $K$-valued polynomial functions on $N_{\mathbb{Q}}$,
where $\{t_1, \cdots, t_n\} \subset \hom_{\mathbb{Z}}(N, \mathbb{Z})$ is a basis.
Let $\mathbf{m} = (t_1, \cdots, t_n) \subset K\left[\mathbf{t}\right]$. 
Suppose that $\Sigma$ is a simplicial fan. 
\begin{enumerate}[(a)]
\item 
Let $\mathcal{A}(\Sigma) = \mathcal{A}_K(\Sigma)$ be the graded $K\left[\mathbf{t}\right]$-algebra of all conewise polynomial functions on $\Sigma$, i.e. $K$-valued functions on $|\Sigma|$ which restrict to polynomials on cones of $\Sigma$. 
The grading on $\mathcal{A}(\Sigma)$ is given by degree.
More precisely, $f \in \mathcal{A}^k(\Sigma)$ if and only if $f|_{\sigma}$ is a polynomial of degree $k$ for each $\sigma \in \Sigma$.

\item
Define $H(\Sigma) = H_K(\Sigma) \defeq \mathcal{A}(\Sigma) / \mathbf{m} \mathcal{A}(\Sigma)$.

\item
We note that $l \in \mathcal{A}^1_{\mathbb{Q}}(\Sigma)$ is strictly convex if and only if $l+f$ is strictly convex for each $f \in \operatorname{hom}_{\mathbb{Q}}(N_{\mathbb{Q}}, \mathbb{Q}) = \mathfrak{m}_1 = \sum \mathbb{Q} t_i$. So it makes sense to say whether a class in $H_{\mathbb{Q}}^1(\Sigma)$ is strictly convex or not.
\end{enumerate}
\end{definition}

\begin{remark}
Let $\boldsymbol{\Sigma} = (N, \Sigma, \{v_{\rho}\}_{\rho \in \Sigma(1)})$ be a stacky fan.
The Stanley-Reisner ring of $\Sigma$ is defined to be 
$$
\operatorname{SR}[\Sigma] \defeq K\left[x_{\rho}\right]_{\rho \in \Sigma(1)}/\left(x_{\rho_1} \ldots x_{\rho_r}\middle|
\text{$\rho_1, \ldots, \rho_r$ do not generate a cone in $\Sigma$}\right).
$$
Then we have an isomorphism
$$\begin{aligned}
\operatorname{SR}[\Sigma] &\xrightarrow{\sim} \mathcal{A}(\Sigma),\\
x_{\rho} &\mapsto \chi_{\rho},
\end{aligned}$$
where $\chi_{\rho} \in \mathcal{A}^1(\Sigma)$ is the unique conewise linear function such that 
$$
\chi_{\rho}(v_{\rho^{\prime}}) = \begin{cases}
1, & \rho^{\prime} = \rho,\\
0, & \rho^{\prime} \neq \rho.\\
\end{cases}
$$
For details, see \cite[Theorem 4.2]{billera1992modules}.
\end{remark}

\begin{definition}\label{def:h-vector}
The $f$-vector of a fan $\Sigma$ is the sequence $\left(f_{-1}, f_0, \ldots, f_{n-1}\right)$ where $f_i = |\Sigma(i+1)|$. The $f$-polynomial is
$$
f(t) \defeq f_{-1} t^n+f_0 t^{n-1}+\cdots+f_{n-2} t+f_{n-1} .
$$
The $h$-polynomial is the polynomial given by
$$
h(t)=f(t-1) .
$$
The $h$-vector is the sequence $\left(h_0, h_1, \ldots, h_n\right)$ of coefficients of $h(t)$ :
$$
h(t)=h_0 t^n+h_1 t^{n-1}+\cdots+h_{n-1} t+h_n.
$$
\end{definition}

Let $\Sigma$ be a simplicial fan.
By \cite[Corallory 4.10.]{billera1989algebra},
$\mathcal{A}(\Sigma)$ is a free $K\left[\mathbf{t}\right]$-module 
and a basis for $\mathcal{A}(\Sigma)$ contains $h_i$ elements of degree $i$.
As $H(\Sigma) = \mathcal{A}(\Sigma) \otimes_{K[\mathbf{t}]} K[\mathbf{t}]/\mathbf{m}$, we have $\dim H^i(\Sigma) = h_i$.
In particular, we know that $H^i(\Sigma) = 0$, for any $i > n$, and $\dim H^n(\Sigma) = 1$ if $\Sigma$ is complete. (See e.g. \cite[Theorem 12.5.9]{cox2011toric}.)
In fact, we have a so-called ``evaluation map'' $\langle\cdot\rangle: H^n(\Sigma) \xrightarrow{\sim} K$. (For specific definition, see \cite[Theorem 2.2]{brion1997structure}, also \cite[Section 2.3]{fleming2010hard}.)
We will also use $\langle\cdot\rangle$ to denote the composition of the projection map $H(\Sigma) \to H^n(\Sigma)$ and the evaluation map.

\begin{theorem}
\label{thm:PMHS-SR-ring}
Let $l$ be a strictly convex conewise linear function on a complete simplicial fan $\Sigma$. 
Consider
\begin{itemize}
\item an increasing filtration $W_{\bullet}$ on $H(\Sigma)$ given by $W_{2k} = W_{2k+1} \defeq \bigoplus_{i \leq k} H^{n-i}(\Sigma)$,
\item a decreasing filtration $F^{\bullet}$ on $H(\Sigma)$ given by $F^k \defeq \bigoplus_{i \geq k} H^{n-i}(\Sigma)$,
\item the linear transformation on $H(\Sigma)$ given by the multiplication by $l$,
\item a bilinear form $Q = Q_{\Sigma}$ on $H(\Sigma)$ such that $Q(h_1, h_2) \defeq (-1)^{k_1}\left\langle h_1 \cdot h_2 \right\rangle$, for any $h_i \in H^{k_i}(\Sigma)$.
\end{itemize}
Then the tuple $\left(H_{\mathbb{Q}}(\Sigma), W_{\bullet}, F^{\bullet}, l, Q\right)$ is a polarized mixed Hodge structure of Hodge-Tate type and with weight $n$.
(For the definition of polarized mixed Hodge structures, see e.g. \cite[Definition 10.16.]{hertling2002frobenius}.
For the definition of Hodge-Tate type, see e.g. \cite[p.5]{sabbah2018some}.
)
\end{theorem}
\begin{proof}
\begin{enumerate}[(i)]
\item 
Since $W_{2k} = W_{2k+1}$ and $H(\Sigma) = F^{k+1} \oplus W_{2k}$ for all $k$, we know that $\left(H(\Sigma), W_{\bullet}, F^{\bullet}\right)$ forms mixed Hodge structure of Hodge-Tate type.

\item
\begin{enumerate}
\item 
Since $l \in H^{1}(\Sigma)$, we know that $l(H^i) \subset H^{i+1}$, i.e. $l$ is a map of degree $(-1,-1)$ of $\left(H(\Sigma), W_{\bullet}, F^{\bullet}\right)$.

\item 
Since $l^{n+1} \in H^{n+1}(\Sigma) = 0$, we know that $l$ is nilpotent.

\item
By \cite[Theorem 7.3.]{mcmullen1993simple} or \cite[Theorem 1.1.]{fleming2010hard}, multiplication by
$$
l^{n-2 k}: 
\operatorname{Gr}^W_{n + (n-2k)} = 
H^k(\Sigma) \rightarrow 
\operatorname{Gr}^W_{n - (n-2k)} = 
H^{n-k}(\Sigma)
$$
is an isomorphism for each $k$.
Therefore, $W_{\bullet} = M(l)_{\bullet-n}$, where $M(l)$ is the monodromy filtration of $l$.
\end{enumerate}

\item
\begin{enumerate}
\item 
Note that for $h_i \in H^{k_i}(\Sigma)$, $Q(h_1, h_2) \neq 0$ only if $k_1 + k_2 = n$, i.e. 
$$Q\left(F^k, F^{n-k+1}\right)=0.$$

\item
Furthermore, when $k_1 + k_2 = n$, we have 
$$
Q(h_1, h_2) = (-1)^{k_1}\left\langle h_1 \cdot h_2 \right\rangle = (-1)^{n} (-1)^{k_2}\left\langle h_1 \cdot h_2 \right\rangle = (-1)^{n} Q(h_2, h_1).
$$
Therefore $Q$ is $(-1)^n$-symmetric.

\item
For $h_i \in H^{k_i}(\Sigma)$, $Q(l h_1, h_2) + Q(h_1, l h_2) = \left((-1)^{k_1} + (-1)^{k_1+1}\right)\left\langle l \cdot h_1 \cdot h_2 \right\rangle = 0$.

\item
Note that
$$
PH_{n + \ell} (\Sigma) = \begin{cases}
\operatorname{ker}\left(l^{n-2 k+1}: H^k(\Sigma) \rightarrow H^{n-k+1}(\Sigma)\right), & \ell = n - 2k,\\
0, & \ell = n - 2k -1.
\end{cases}
$$
Set $\ell = n - 2k$.
The pure Hodge structure on $PH_{n + \ell} (\Sigma)$ is given by
$H^{n-k,n-k} = PH_{n + \ell} (\Sigma)$.

By \cite[Theorem 8.2.]{mcmullen1993simple} or \cite[Theorem 1.2.]{fleming2010hard}, 
the quadratic form 
$$h \mapsto (-1)^k \left\langle l^{\ell} \cdot h \cdot h \right\rangle$$
is positive definite on $PH_{n + \ell}$.
Therefore, we know that 
$$
i^{2 p-n-\ell} Q(h, l^{\ell}\bar{h})>0
$$ 
if $h \in F^p PH_{n + \ell} (\Sigma) \cap \overline{F^{n+\ell-p} PH_{n + \ell} (\Sigma)}$, $h \neq 0$.
\end{enumerate}
\end{enumerate}
\end{proof}

\subsection{Quotient stacky fans}

\begin{definition}
\label{def:quotient-stacky-fan}
Let $\boldsymbol{\Sigma} = (N, \Sigma, \{v_{\rho}\}_{\rho \in \Sigma(1)})$ be a stacky fan.
Fix a cone $\sigma$ in the fan $\Sigma$. 
\begin{enumerate}[(a)]
\item
We define
$$
\begin{aligned}
& \operatorname{Star}_{\Sigma}(\sigma)=\{\delta \in \Sigma \mid \sigma \prec \delta\}, \\
& \overline{\operatorname{Star}}_{\Sigma}(\sigma)=\left\{\tau \in \Sigma \mid \tau \prec \delta \text { for some } \delta \in \operatorname{Star}(\sigma)\right\}, \\
& \operatorname{Link}_{\Sigma}(\sigma)=\left\{\tau \in \overline{\operatorname{Star}}(\sigma) \mid \tau \cap \sigma=0\right\} .
\end{aligned}
$$
And 
$$\begin{aligned}
\overline{\boldsymbol{\operatorname{Star}}}_{\boldsymbol{\Sigma}}(\sigma) &= \left(N, \overline{\operatorname{Star}}_{\Sigma}(\sigma), \{v_{\rho}\}_{\rho \in \overline{\operatorname{Star}}_{\Sigma}(\sigma)(1)}\right),\\
\boldsymbol{\operatorname{Link}}_{\boldsymbol{\Sigma}}(\sigma) &= \left(N, \operatorname{Link}_{\Sigma}(\sigma), \{v_{\rho}\}_{\rho \in \operatorname{Link}(\sigma)(1)}\right).
\end{aligned}$$

\item 
Let $N_\sigma$ be the subgroup of $N$ generated by the set $\left\{v_{\rho} | \rho \in \sigma(1)\right\}$ and let $N(\sigma)$ be the quotient group $N / N_\sigma$. 

\item
The quotient fan $\Sigma(\sigma)$ in $N(\sigma)_{\mathbb{Q}}$ is the set 
$$
\Sigma(\sigma) 
\defeq \left\{\tau+\left(N_\sigma\right)_{\mathbb{Q}} \subset N(\sigma)_{\mathbb{Q}} \middle| \tau \in \operatorname{Star}(\sigma)\right\}.
$$

\item 
The quotient stacky fan $\boldsymbol{\Sigma}(\sigma)$ is the triple $\left(N(\sigma), \Sigma( \sigma), \left\{v_{\rho} + N_{\sigma}\right\}_{\rho \in \operatorname{Link}(\sigma)(1)}\right)$.
\end{enumerate}
\end{definition}

Note that we have the following maps of stacky fans 
$$\xymatrix{
\overline{\boldsymbol{\operatorname{Star}}}_{\boldsymbol{\Sigma}}(\sigma) \ar[d]_{\pi} \ar@{^(_->}[r]^i & \boldsymbol{\Sigma}\\
\boldsymbol{\Sigma}(\sigma)
}$$
Hence we have maps $\mathcal{A}(\boldsymbol{\Sigma}) \xrightarrow{i^*} \mathcal{A}(\overline{\boldsymbol{\operatorname{Star}}}_{\boldsymbol{\Sigma}}(\sigma)) \xleftarrow{\pi^*} \mathcal{A}(\boldsymbol{\Sigma}(\sigma))$ and $H(\boldsymbol{\Sigma}) \xrightarrow{i^*} H(\overline{\boldsymbol{\operatorname{Star}}}_{\boldsymbol{\Sigma}}(\sigma)) \xleftarrow{\pi^*} H(\boldsymbol{\Sigma}(\sigma))$.
In fact, 
$\pi^* : H(\boldsymbol{\Sigma}(\sigma)) \rightarrow H(\overline{\boldsymbol{\operatorname{Star}}}_{\boldsymbol{\Sigma}}(\sigma))$ is an isomorphism. 
Moreover, a conewise linear function $l \in H^1(\boldsymbol{\Sigma}(\sigma))$ is strictly convex if and only if $\pi^*(l) \in H^1(\overline{\boldsymbol{\operatorname{Star}}}_{\boldsymbol{\Sigma}}(\sigma))$ is strictly convex.
(See \cite[Section 1.2]{gross2011tropical}.)
As a consequence, we have
\begin{corollary}
\label{lemma:star-quotient-isom}
Let $l$ be a strictly convex conewise linear function on a complete simplicial fan $\Sigma$. 
For any $\sigma \in \Sigma$,
$\left(H_{\mathbb{Q}}(\overline{\operatorname{Star}}(\sigma)), W_{\bullet}, F^{\bullet}, l, Q_{\Sigma(\sigma)}\right)$ is a polarized mixed Hodge structure of Hodge-Tate type with weight $\operatorname{codim} \sigma \defeq n - \dim \sigma$.
\end{corollary}

\section{Gauss-Manin system and Brieskorn lattice}
\label{section:Brieskorn}

\subsection{Twisted algebraic de Rham complex}

Let $f$ be a non-degenerate Laurent polynomial in $K\left[\mathbf{t}^{ \pm 1}\right]$ such that $P$ is the Newton polyhedron of $f$ at $\infty$,
let $\theta$ be a new variable, and let $\tau = \theta^{-1}$. 
The twisted algebraic de Rham complex attached to $f$ is the complex of $K\left[\tau^{\pm 1}\right]$-modules
$$
\left(\Omega^{\bullet}(U)\left[\tau^{\pm 1}\right], \mathrm{e}^{\tau f} \circ \mathrm{d} \circ \mathrm{e}^{- \tau f}\right),
$$
where
$$
\mathrm{e}^{\tau f} \circ \mathrm{d} \circ \mathrm{e}^{- \tau f} = \mathrm{d} - \tau \mathrm{d} f \wedge.
$$
We define $\Omega(f)$ to be the complex
$$
\Omega(f) \defeq \left(\Omega^{\bullet}(U)\left[\tau^{\pm 1}\right], \theta \mathrm{d}-\mathrm{d} f \wedge\right).
$$
Define a connection $\nabla$ on $\Omega(f)$ by
\begin{align}\label{align:partial_theta}
\nabla_{\partial_{\tau}} = \mathrm{e}^{\tau f} \circ \partial_{\tau} \circ \mathrm{e}^{- \tau f} = \frac{\partial}{\partial \tau} - f.
\end{align}
Consider the following complex of $K[\theta]$-modules:
$$
\Omega_0(f) \defeq \left(\Omega^{\bullet}(U)\left[\theta\right], \theta \mathrm{d}-\mathrm{d} f \wedge\right).
$$
Then $\Omega_0(f) \otimes_{K[\theta]} K\left[\tau^{\pm 1}\right] \xrightarrow{\sim} \Omega(f)$,
and $\Omega_0(f) \otimes_{K[\theta]} \left(K\left[\theta\right]/\theta K\left[\theta\right]\right)$ isomorphic to the Koszul complex 
$$
K(f) \defeq \left(\Omega^{\bullet}(U), - \mathrm{d} f \wedge\right).
$$
Endow $\Omega(f)$ with the increasing filtration $\Phi_{\bullet}$ by $\Phi_p\Omega(f) \defeq \theta^{-p} \Omega_0(f)$. 
We have
$$
\operatorname{Gr}^{\Phi}_p \Omega(f) \cong \left.\theta^{-p}\Omega_0(f)\middle/\theta^{-p+1} \Omega_0(f) \cong K(f).\right.
$$
for all $p$.
The Fourier transform of the Gauss-Manin system is defined to be
$$
G \defeq H^n\left(\Omega(f)\right) = \left.\Omega^n(U)\left[\tau^{\pm 1}\right] \middle/(\mathrm{d}-\tau \mathrm{d} f \wedge) \Omega^{n-1}(U)\left[\tau^{\pm 1}\right] .\right.
$$
The operator $\nabla_{\partial_{\tau}}$ acts on $G$.
The Brieskorn lattice is defined to be
$$
G_0 \defeq H^n\left(\Omega_0(f)\right) = \left.\Omega^n(U)[\theta]\middle/(\theta \mathrm{d}-\mathrm{d} f \wedge) \Omega^{n-1}(U)[\theta].\right.
$$
The Jacobian ring is defined to be
\begin{align}
\label{align:def-Jaco}
J_f \defeq \left.K\left[\mathbf{t}^{\pm 1}\right]\middle/\left(t_1 \frac{\partial f}{\partial t_1}, \cdots, t_n \frac{\partial f}{\partial t_n}\right)\right..
\end{align}
We have $J_f \cong H^n\left(K(f)\right)$.
Consider the following commutative diagram:
\begin{align}
\label{align:three-complexes}
\xymatrix{
\cdots \ar[r] & \Omega^{n-1}(U) \ar[r]^{-\mathrm{d} f \wedge } & \Omega^n(U) \ar[r]^{\overline{\epsilon}} & J_f \ar[r] & 0\\
\cdots \ar[r] & \Omega^{n-1}(U)\left[\theta\right] \ar[r]^{\theta \mathrm{d}-\mathrm{d} f \wedge } \ar[d] \ar[u] & \Omega^n(U)\left[\theta\right] \ar[r]^{\epsilon_0} \ar[d]^{i} \ar[u]_{p} & G_0 \ar[r] \ar[d]^{\tilde{i}} \ar[u]_{\tilde{p}} & 0\\
\cdots \ar[r] & \Omega^{n-1}(U)\left[\tau^{ \pm 1}\right] \ar[r]^{\theta \mathrm{d}-\mathrm{d} f \wedge } & \Omega^n(U)\left[\tau^{ \pm 1}\right] \ar[r]^{\epsilon} & G \ar[r] & 0
}\end{align}
where $p, i, \tilde{p}, \tilde{i}, \overline{\epsilon}, \epsilon_0$ and $\epsilon$ are canonical morphisms.
The three horizontal lines in the above commutative diagram are three complexes, which we will call $\tilde{K}(f)$, $\tilde{\Omega}_0(f)$, $\tilde{\Omega}(f)$ from top to bottom.

\subsection{The Newton filtration}
\label{subsection:The_Newton_filtration}

We define the Newton filtration $\mathcal{N}_\alpha$ ($\alpha \in \mathbb{Q}$) on $K\left[\mathbf{t}^{ \pm 1}\right]$ by
$$
\mathcal{N}_\alpha K\left[\mathbf{t}^{ \pm 1}\right]
= \operatorname{span}\left\{\mathbf{t}^u | \deg(u) \leq \alpha\right\},
$$
where 
\begin{align}\label{def_deg}
\deg(u) = \deg_P(u) \defeq \min\{\lambda | u \in \lambda P\}
\end{align}
is the strictly convex conewise linear function associated to $P$.
Later we will view $\deg$ as an element in $\mathcal{A}^1(\Sigma_P)$.

The Newton filtration on $\Omega^k(U)$ is defined by
$$
\mathcal{N}_\alpha \Omega^k(U) \defeq \sum_{i_1 < \cdots < i_k} \mathcal{N}_{\alpha + k - n} K\left[\mathbf{t}^{\pm 1}\right] \cdot \frac{\mathrm{d} t_{i_1}}{t_{i_1}} \wedge \cdots \wedge \frac{\mathrm{d} t_{i_k}}{t_{i_k}}.
$$
Extend it to $\Omega^k(U)\left[\tau^{\pm 1}\right]$ by
\begin{align}
\label{align:newton-fil-diff-tau}    
\mathcal{N}_\alpha \Omega^k(U)[\tau^{\pm 1}] \defeq \sum_{i \in \mathbb{Z}} \tau^i \mathcal{N}_{\alpha+i} \Omega^k(U).
\end{align}
It induces a filtration $\mathcal{N}_{\bullet}$ on $\Omega^k(U)\left[\theta\right]$.
More precisely, we have
$$
\mathcal{N}_\alpha \Omega^k(U)[\theta] = \mathcal{N}_\alpha \Omega^k(U)+\theta \mathcal{N}_{\alpha-1} \Omega^k(U) + \cdots + \theta^i \mathcal{N}_{\alpha-i} \Omega^k(U)+\cdots.
$$

Define the filtration on the complex $\Omega(f)$ by
$$\begin{aligned}
\mathcal{N}_\alpha\Omega(f) &\defeq \left(\cdots \to \mathcal{N}_{\alpha}\Omega^{n-1}(U)\left[\tau^{\pm 1}\right] \xrightarrow{\theta \mathrm{d} - \mathrm{d} f \wedge} \mathcal{N}_\alpha\Omega^{n}(U)\left[\tau^{\pm 1}\right] \to 0\right).
\end{aligned}$$
Let
$$\begin{aligned}
\mathcal{N}_{\alpha}G &\defeq H^n(\mathcal{N}_{\alpha}\Omega(f)), \quad
\mathcal{N}_{<\alpha}G &\defeq H^n(\mathcal{N}_{<\alpha}\Omega(f)).
\end{aligned}$$
We will show that $\mathcal{N}_{\alpha}G$ and $\mathcal{N}_{<\alpha}G$ are sub-modules of $G$ (see Lemma \ref{cor:strict}),
but at this moment, we only have canonical maps 
$\mathcal{N}_{\beta}G \to \mathcal{N}_{< \alpha}G \to \mathcal{N}_{\alpha}G \to G$
for all $\beta < \alpha$.
Denote $\operatorname{Gr}^{\mathcal{N}}_{\alpha}G \defeq \operatorname{coker}\left(\mathcal{N}_{< \alpha}G \to \mathcal{N}_{\alpha}G\right)$.

We define $\mathcal{N}_{\bullet}\Omega_0(f)$, $\mathcal{N}_{\bullet}K(f)$, $\mathcal{N}_{\bullet}G_0$ and $\mathcal{N}_{\bullet}J_f$ in the same way.

\begin{lemma}
\label{lemma:exact-gr-N}
Fix notations by the following diagram:
\begin{align}
\label{align:three-complexes-gr}
\xymatrix{
\operatorname{Gr}^{\mathcal{N}}_{\alpha}\tilde{K}(f): & \cdots \ar[r] & \operatorname{Gr}^{\mathcal{N}}_{\alpha}\Omega^{n-1}(U) \ar[r]^{-\mathrm{d} f \wedge } & \operatorname{Gr}^{\mathcal{N}}_{\alpha}\Omega^n(U) \ar[r]^{\overline{\epsilon}} & \operatorname{Gr}^{\mathcal{N}}_{\alpha}J_f \ar[r] & 0\\
\operatorname{Gr}^{\mathcal{N}}_{\alpha}\tilde{\Omega}_0(f): & \cdots \ar[r] & \operatorname{Gr}^{\mathcal{N}}_{\alpha}\Omega^{n-1}(U)\left[\theta\right] \ar[r]^{\theta \mathrm{d}-\mathrm{d} f \wedge } \ar[d] \ar[u] & \operatorname{Gr}^{\mathcal{N}}_{\alpha}\Omega^n(U)\left[\theta\right] \ar[r]^{\epsilon_0} \ar[d] \ar[u] & \operatorname{Gr}^{\mathcal{N}}_{\alpha}G_0 \ar[r] \ar[d] \ar[u] & 0\\
\operatorname{Gr}^{\mathcal{N}}_{\alpha}\tilde{\Omega}(f): & \cdots \ar[r] & \operatorname{Gr}^{\mathcal{N}}_{\alpha}\Omega^{n-1}(U)\left[\tau^{ \pm 1}\right] \ar[r]^{\theta \mathrm{d}-\mathrm{d} f \wedge } & \operatorname{Gr}^{\mathcal{N}}_{\alpha}\Omega^n(U)\left[\tau^{ \pm 1}\right] \ar[r]^{\epsilon} & \operatorname{Gr}^{\mathcal{N}}_{\alpha}G \ar[r] & 0
}\end{align}
The horizontal lines are exact.
Furthermore, we have
\begin{align}
\label{align:gr-G-on-grade-N-G}
\operatorname{Gr}^{\Phi}_{p} H^n\left(\operatorname{Gr}^\mathcal{N}_\alpha\Omega(f)\right) &\cong \operatorname{Gr}^\mathcal{N}_{\alpha+p}J(f) \tau^{p},
\end{align}
for any $p$.
\end{lemma}
\begin{proof}
By \cite[Theorem 2.8.]{kouchnirenko1976polyedres}, we have 
$
H^i\left(\operatorname{Gr}^{\mathcal{N}}K(f)\right) = 0
$
for all $i \neq n$.
As a consequence, we have a commuatative diagram
$$\xymatrix{
H^{n-1}\left(\operatorname{Gr}^{\mathcal{N}}_{\alpha}K(f)\right) \ar[r] \ar@{=}[d] &
H^{n}\left(\mathcal{N}_{<\alpha}K(f)\right) \ar[r] \ar@{=}[d] &
H^{n}\left(\mathcal{N}_{\alpha}K(f)\right) \ar[r] \ar@{=}[d] &
H^n\left(\operatorname{Gr}^{\mathcal{N}}_{\alpha}K(f)\right) \ar[r] \ar@{-->}[d] &
0\\
0 \ar[r] &
\mathcal{N}_{<\alpha}J_f \ar[r] &
\mathcal{N}_{\alpha}J_f \ar[r] &
\operatorname{Gr}^{\mathcal{N}}_{\alpha}J_f \ar[r] &
0
}$$
where the first horizontal line is a exact sequence.
So $H^n\left(\operatorname{Gr}^{\mathcal{N}}_{\alpha}K(f)\right) \cong \operatorname{Gr}^{\mathcal{N}}_{\alpha}J_f$.
It follows that $\operatorname{Gr}^{\mathcal{N}}_{\alpha}\tilde{K}(f)$ is exact.
Note that
$$\begin{aligned}
\operatorname{Gr}^{\mathcal{N}}_{\alpha}\Omega^{k}(U)\left[\tau^{\pm 1}\right] 
&\cong \bigoplus_{-\alpha \leq i} \tau^i \operatorname{Gr}^{\mathcal{N}}_{\alpha+i}\Omega^{k}(U) \\
\Phi_p\operatorname{Gr}^{\mathcal{N}}_{\alpha}\Omega^{k}(U)\left[\tau^{\pm 1}\right] 
&\cong \bigoplus_{-\alpha \leq i \leq p} \tau^i \operatorname{Gr}^{\mathcal{N}}_{\alpha+i}\Omega^{k}(U).
\end{aligned}$$
Hence the filtration $\Phi_{\bullet}$ on the complex $\operatorname{Gr}^{\mathcal{N}}_\alpha\Omega(f)$ is bounded below and exhaustive
and
$$\begin{aligned}
\operatorname{Gr}^{\Phi}_{p}\operatorname{Gr}^\mathcal{N}_\alpha\Omega(f) &\cong \left(\cdots \to \operatorname{Gr}^\mathcal{N}_{\alpha+p}\Omega^{n-1}(U)\tau^{p} \xrightarrow{- \mathrm{d} f \wedge} \operatorname{Gr}^\mathcal{N}_{\alpha+p}\Omega^{n}(U)\tau^{p} \to 0\right)\\
&\cong \operatorname{Gr}^{\mathcal{N}}_{\alpha + p}K(f) \tau^{p}.
\end{aligned}$$
Therefore we have a spectral sequence
\begin{align}
\label{align:spectral-sequence-G-on-grade-N}
E_1^{pq} = 
H^{p+q}\left(\operatorname{Gr}^{\mathcal{N}}_{\alpha-p}K(f)\right)\tau^{-p} 
\Rightarrow H^{p+q}(\operatorname{Gr}^\mathcal{N}_\alpha\Omega(f))
\end{align}
Thus 
$
H^i\left(\operatorname{Gr}^\mathcal{N}_\alpha\Omega(f)\right) = 0
$
for all $i \neq n$,
and we get 
$$\begin{aligned}
\operatorname{Gr}^{\Phi}_{p} H^n\left(\operatorname{Gr}^\mathcal{N}_\alpha\Omega(f)\right) 
\cong H^n\left(\operatorname{Gr}^\mathcal{N}_{\alpha + p}K(f)\tau^p\right) 
\cong \operatorname{Gr}^\mathcal{N}_{\alpha+p}J(f) \tau^{p},
\end{aligned}$$
where the last isomorphism comes from the discussion at the beginning.
Hence $\operatorname{Gr}^\mathcal{N}_\alpha\tilde{\Omega}(f)$ is exact.

Similarly we can show that $\operatorname{Gr}^{\mathcal{N}}_\alpha\tilde{\Omega}_0(f)$ is exact.
\end{proof}

\begin{lemma}[\protect{\cite[Lemma 4.3]{kouchnirenko1976polyedres}}]
\label{lemma:exact-to-strict}
Let $A$ be a ring. Let 
$$
(L, F_{\bullet}) \xrightarrow{g} (M, F_{\bullet}) \xrightarrow{f} (N, F_{\bullet})
$$
be a complex of filtered $A$-modules.
Assume that the index of $F_{\bullet}$ is discrete, $F_{\bullet}$ is exhaustive on $M$ and
$$
\operatorname{Gr}^F L \xrightarrow{g} \operatorname{Gr}^F M \xrightarrow{f} \operatorname{Gr}^F N
$$
is exact.
Then $f$ is strict, i.e. $f(M) \cap F_{\alpha} N = f(F_{\alpha} M)$, for all $\alpha$.
\end{lemma}
\begin{proof}
For any $f(m) \in f(M) \cap F_{\alpha} N$,
as $F_{\bullet}$ is exhaustive on $M$,
$m \in F_{\beta}M$ for some $\beta$.
If $\beta > \alpha$, 
then $f([m]) = 0 \in \operatorname{Gr}^F_{\beta} N$.
Hence there exists $l \in F_{\beta} L$, such that $[m] = g([l]) \in \operatorname{Gr}^F_{\beta} M$, i.e. $m - g(l) \in F_{< \beta} M$.
Thus $f(m) = f(m - g(l)) \in f\left(F_{< \beta} M\right)$.
As the index of $F_{\bullet}$ is discrete, by induction, we know that $f(m) \in f\left(F_{\alpha} M\right)$.
\end{proof}

\begin{lemma}
\label{cor:strict}
\begin{enumerate}[(i)]
\item 
$\mathcal{N}_{\alpha} J_f$ ($\mathcal{N}_{\alpha} G_0$, $\mathcal{N}_{\alpha} G$, respectively) are submodules of $J_f$ ($G_0$, $G$, respectively) for all $\alpha$
and all the morphisms in (\ref{align:three-complexes}) are strict with respect to $\mathcal{N}_{\bullet}$.
\item The three horizontal lines in (\ref{align:three-complexes}) are exact.
\end{enumerate}
\end{lemma}
\begin{proof}
\begin{enumerate}[(i)]
\item 
By Lemma \ref{lemma:exact-gr-N}, and Lemma \ref{lemma:exact-to-strict}, we know that all the boundary operators (i.e. $- \operatorname{d} f \wedge$ in the first horizontal line and $\theta \operatorname{d} - \operatorname{d} f \wedge$ in the second and the third horizontal lines) in (\ref{align:three-complexes}) are strict.
Therefore we know that 
$$\begin{aligned}
\mathcal{N}_{\alpha} J_f 
&= \mathcal{N}_{\alpha}\Omega^n(U)/\operatorname{d} f \wedge \mathcal{N}_{\alpha}\Omega^{n-1}(U)\\
&= \mathcal{N}_{\alpha}\Omega^n(U)/\left(\mathcal{N}_{\alpha}\Omega^n(U) \cap (\operatorname{d} f \wedge \Omega^{n-1}(U))\right)\\
&\cong\operatorname{im}\left(\mathcal{N}_{\alpha}\Omega^n(U) \to J_f\right) \subset J_f.
\end{aligned}$$
Similarly, we have
$$
\mathcal{N}_{\alpha} G_0
\cong \operatorname{im}\left(\mathcal{N}_{\alpha}\Omega^n(U)[\theta] \to G_0\right),
\quad
\mathcal{N}_{\alpha} G
\cong \operatorname{im}\left(\mathcal{N}_{\alpha}\Omega^n(U)\left[\tau^{\pm 1}\right] \to G\right).
$$
Directly from their definitions, we can see that other morphisms in (\ref{align:three-complexes}) are also strict.

\item
Note that filtrations $\mathcal{N}_{\bullet}$ on complexes $\Omega(f), \Omega_0(f), K(f)$ are bounded below and exhaustive.
Hence spectral sequences associate to them converge.
By Lemma \ref{lemma:exact-gr-N}, all of them collapse.
Therefore we know that $\tilde{\Omega}(f), \tilde{\Omega}_0(f), \tilde{K}(f)$ are exact.
\end{enumerate}
\end{proof}

\begin{remark}
\label{rmk:newton-fil-def}
As
$$\begin{aligned}
\mathcal{N}_{\alpha}\Omega^n(U)\left[\tau^{\pm 1}\right] 
= \sum_{k \geq 0} \tau^{k} i\left(\mathcal{N}_{\alpha+k}\Omega^n(U)\left[\theta\right]\right).
\end{aligned}$$
We have
$$\begin{aligned}
\mathcal{N}_\alpha G 
&= \epsilon\left(\mathcal{N}_{\alpha}\Omega^n(U)\left[\tau^{\pm 1}\right]\right)
= \epsilon\left(\sum_{k \geq 0} \tau^{k} i\left(\mathcal{N}_{\alpha+k}\Omega^n(U)\left[\theta\right]\right)\right)\\
&= \sum_{k \geq 0} \tau^{k} (\epsilon\circ i)\left(\mathcal{N}_{\alpha+k}\Omega^n(U)\left[\theta\right]\right)
= \sum_{k \geq 0} \tau^{k} (\tilde{i}\circ \epsilon_0)\left(\mathcal{N}_{\alpha+k}\Omega^n(U)\left[\theta\right]\right)\\
&= \sum_{k \geq 0} \tau^{k} \tilde{i}\left(\mathcal{N}_{\alpha+k} G_0\right)
= \mathcal{N}_\alpha G_0+\tau \mathcal{N}_{\alpha+1} G_0+\cdots+\tau^k \mathcal{N}_{\alpha+k} G_0+\cdots
\end{aligned}$$
Therefore the filtrations $\mathcal{N}_{\bullet}$ on $J_f, G_0$ and $G$ defined above coincide with those in \cite[Section 4.a.]{douai2003gauss}.
\end{remark}

\subsection{The vanishing cycle}
\begin{definition}\label{def:vanishing-cycle}
Let $H_\alpha=\operatorname{Gr}_\alpha^{\mathcal{N}}(G)$. 
\begin{enumerate}[(a)]
\item 
Let
$$
\nu = \begin{cases}
n, & \alpha = 0,\\
n-1, & 0 < \alpha < 1.
\end{cases}
$$

\item
The filtration $\Phi_{\bullet}$ on $G$ induces a filtration $\Phi_{\bullet}$ on $H_{\alpha}$.
Define the Hodge filtration on $H_{\alpha}$ to be $F^{\bullet} H_\alpha \defeq \Phi_{\nu - \bullet} H_{\alpha}$.

\item
Let $N \defeq - (\tau \nabla_{\partial_\tau} + \alpha)$. It is a nilpotent endomorphism on $H_{\alpha}$ (see \cite[Lemma 12.2]{sabbah1999hypergeometric}).
Define the weight filtration on $H_{\alpha}$ to be $W_{\bullet} = M(N)_{\bullet - \nu}$, where $M(N)$ denotes the monodromy filtration of $N$.
\end{enumerate}
\end{definition}
\begin{remark}
Note that the Newton filtration $\mathcal{N}_{\bullet}G$ is equal to the Malgrange–Kashiwara filtration $V_{\bullet}G$. 
See \cite[Lemma 12.2.]{sabbah1999hypergeometric}. For the definition of $V_{\bullet}G$, see  \cite[p178]{sabbah1999hypergeometric}.
Therefore we can also write $H_\alpha=\operatorname{Gr}_\alpha^{V}(G)$. 
\end{remark}

Denote
$$
H = \bigoplus_{\alpha \in[0,1)} H_\alpha, \quad 
H_{\neq 0} \defeq \bigoplus_{\alpha \in(0,1)} H_\alpha.
$$
Then we have $N$, $F^{\bullet}$, $W_{\bullet}$ on $H_0$ (on $H_{\neq 0}$, respectively).
When $K = \mathbb{C}$, we know that they underlie a polarized mixed Hodge structure of weight $n$ (of weight $n-1$, respectively). 
See \cite[p4]{sabbah2018some}, \cite[p187]{hertling2002frobenius},\cite[p215]{sabbah1999hypergeometric},\cite[6.5]{scherk1985mixed}.
By \cite[Remark 3.8]{saito1989structure} (or \cite{scherk1985mixed}), we know that $N^j : (H_{\alpha}, F^{\bullet}) \to (H_{\alpha}, F^{\bullet-j})$ are strict morphisms for any $j \geq 0, \alpha \in [0,1)$.
We have the following result:
\begin{lemma}[\protect{\cite[Proposition 3.7]{saito1989structure}}]
\label{lemma:van-jac-rel}
Let $H$ be a finite-dimensional vector space, $N : H \to H$ a nilpotent linear transformation and $\Phi_{\bullet}$ an increasing filtration such that $N(\Phi_i) \subset \Phi_{i+1}$.

Suppose that $N^j : (H, \Phi_{\bullet}) \to (H, \Phi_{\bullet+j})$ are strict morphisms for any $j \geq 0$.
Then $(H, \Phi_{\bullet}, N)$ are isomorphic to direct sums of the copies of $\left(K[N] / K[N] N^m, F_{\bullet-p}, N\right)$ for some $p \in \mathbb{Z}, m \in \mathbb{N}$, where $F_k K[N] = \operatorname{span}\{1, N, \cdots, N^k\}$.

As a consequence, there exists a spliting $H = \oplus I_i$ such that $\Phi_k = \oplus_{i \leq k} I_i$ and $N(I_i) \subset I_{i+1}$.
In other words, we have a linear isomorphism $H \xrightarrow{\sim} \operatorname{Gr}^{\Phi} H$, such that
$$\begin{aligned}
\xymatrix{
H \ar[r] \ar[d]_{N} & \operatorname{Gr}^{\Phi} H \ar[d]^{N}\\
H \ar[r] & \operatorname{Gr}^{\Phi} H.
}
\end{aligned}$$
is commutative, where $N : \operatorname{Gr}^{\Phi} H \to \operatorname{Gr}^{\Phi} H$ is induced by $\Phi_iH/\Phi_{i - 1}H \xrightarrow{N} \Phi_{i + 1}H/\Phi_{i}H$.
\end{lemma}

Note that by 
Lemma \ref{lemma:exact-gr-N},
we have isomorphisms
$$
\operatorname{Gr}^{\Phi}_{p} H_{\alpha} 
\xrightarrow{\sim} \operatorname{Gr}^\mathcal{N}_{\alpha+p}J(f) \tau^{p}
\xrightarrow[\sim]{\theta^p} \operatorname{Gr}^\mathcal{N}_{\alpha+p}J(f).
$$
By (\ref{align:partial_theta}), we have
$$
\tau \nabla_{\partial_{\tau}} [\omega \tau^k] = [k \omega \tau^k - f \omega \tau^{k+1}].
$$
Hence in $\operatorname{Gr}_p^{\Phi} H_{\alpha}$,
$$
N [\omega \tau^k] = [f \omega \tau^{k+1}].
$$
Therefore we have the following commutative diagram
$$\xymatrix{
\operatorname{Gr}^{\Phi}_{p} H_{\alpha} \ar[r]^{\theta^p}_{\sim} \ar[d]_{N} & \operatorname{Gr}^{\mathcal{N}}_{\alpha+p} J_f \ar[d]^{[f]}\\
\operatorname{Gr}^{\Phi}_{p+1} H_{\alpha} \ar[r]^{\theta^{p+1}}_{\sim} & \operatorname{Gr}^{\mathcal{N}}_{\alpha+p+1} J_f.
}$$
In other words, we have
\begin{align}\label{align:global-H-J-equal}
\left(\operatorname{Gr}^{\Phi}_{\bullet} H_{\alpha}, N\right) 
\xrightarrow{\sim}
\left(\operatorname{Gr}^{\mathcal{N}}_{\alpha+\bullet} J_f, [f]\right).    
\end{align}
Therefore we have 
$$
\left(H_{\alpha}, N\right) 
\xrightarrow{\sim}
\left(\operatorname{Gr}^{\Phi}_{\bullet} H_{\alpha}, N\right) 
\xrightarrow{\sim}
\left(\operatorname{Gr}^{\mathcal{N}}_{\alpha+\bullet} J_f, [f]\right).
$$
Note that the first isomorphism is not canonical.

\section{The graded Jacobian ring}

Consider the Laurent polynomial
\begin{align}
\label{align:def-f}
f = f_{P, \mathbf{a}} \defeq \sum_{v \in P(0)} a_v \mathbf{t}^v \in K[\mathbf{t}^{\pm 1}] = K[t_1^{\pm 1}, \cdots, t_n^{\pm 1}],
\end{align}
where $a_v \in K^*$, for all $v \in P(0)$.

\begin{lemma}\label{lemma:f_conv_non_dege}
$f_{P, \mathbf{a}}$ is convenient and non-degenerate.
\end{lemma}
\begin{proof}
For any face $F$ of $P$, assume the vertices of $F$ are $v_1, \cdots, v_d$, where $v_i = (v_{i1}, \cdots, v_{in}) \in \mathbb{Q}^n = N_{\mathbb{Q}}$.
Then
$$
f_F = a_{v_1} t^{v_1} + \cdots + a_{v_d} t^{v_d}
$$
and
$$
(f_i)_F = v_{1i} a_{v_1} t^{v_1} + \cdots + v_{di} a_{v_d} t^{v_d},
$$
i.e.
$$
\begin{pmatrix}
(f_1)_F\\
\vdots\\
(f_n)_F
\end{pmatrix}
 = 
\begin{pmatrix}
v_{11} & \cdots & v_{d1}\\
\vdots & \ddots & \vdots\\
v_{1n} & \cdots & v_{dn}
\end{pmatrix}
\begin{pmatrix}
a_{v_1} t^{v_1}\\
\vdots\\
a_{v_d} t^{v_d}
\end{pmatrix}.
$$
As $P$ is simiplicial, we know that $v_1, \cdots, v_d$ are linearly independent. Hence $(f_1)_F = \cdots = (f_n)_F = 0$ if and only if $a_{v_j} t^{v_j} = 0$ for all $j$, i.e. $t^{v_j} = 0$ for all $j$.
Therefore $(f_1)_F = \cdots = (f_n)_F = 0$ define an empty subscheme in $U$.
\end{proof}

Keep the notations in Section \ref{section:Polytopes}.
For any $u \in \operatorname{Box}\left(\boldsymbol{\Sigma}_P\right)$,
let
$$
A_{P}(u) \defeq \operatorname{span}\left\{\mathbf{t}^w \middle|w \in P_u(\boldsymbol{\Sigma})\right\} \subset K\left[\mathbf{t}^{ \pm 1}\right].
$$
We have a monomorphism
$$
\operatorname{Gr}^{\mathcal{N}} A_{P}(u)
\hookrightarrow \operatorname{Gr}^{\mathcal{N}} K\left[\mathbf{t}^{ \pm 1}\right].
$$
By (\ref{align:decomp-P_u}), we have 
\begin{align*}
K[\mathbf{t}^{\pm 1}] = \bigoplus_{u \in \operatorname{Box}(\boldsymbol{\Sigma}_P)} A_{P}(u),
\end{align*}
and
\begin{align}
\label{decomposition-Gr-N}
\operatorname{Gr}^{\mathcal{N}} K[\mathbf{t}^{\pm 1}] = \bigoplus_{u \in \operatorname{Box}(\boldsymbol{\Sigma}_P)} \operatorname{Gr}^{\mathcal{N}} A_{P}(u).
\end{align}

\begin{lemma}
Let $a_{\rho} \in K^*$ ($\rho \in \Sigma_P(1)$) and let $u \in \operatorname{Box}(\boldsymbol{\Sigma}_P)$.
Denote $\boldsymbol{\Sigma}_P(u) = \overline{\boldsymbol{\operatorname{Star}}}_{\boldsymbol{\Sigma}_P}(\sigma(u))$.
Let $\mathcal{A}(\boldsymbol{\Sigma}_P(u))$ be the algebra of conewise polynomial functions on $\boldsymbol{\Sigma}_P(u)$. 
(See Definition \ref{def:conewise_polynomial}.)
We have a linear map
\begin{align}
\begin{aligned}
\phi_u = \phi_{\mathbf{a},u} : 
\mathcal{A}(\boldsymbol{\Sigma}_P(u)) 
&\to \operatorname{Gr}^{\mathcal{N}} A_{P}(u)
\end{aligned}
\end{align}
such that the following holds:
\begin{enumerate}
\item 
For $u = 0$, $\phi = \phi_0 : \mathcal{A}(\boldsymbol{\Sigma}_P) \to \operatorname{Gr}^{\mathcal{N}} A_{P}(0)$ is a ring isomorphism.

\item
For general $u$, $\phi_u$ is an isomorphism of $\mathcal{A}(\boldsymbol{\Sigma}_P(u))$-modules, and $\operatorname{Gr}^{\mathcal{N}} A_{P}(u)$ is a free $\mathcal{A}\left(\boldsymbol{\Sigma}_P(u)\right)$-module of rank $1$ with basis $\mathbf{t}^u$. 
We have a commutative diagram
\begin{align}
\label{align:fin}
\xymatrix{
\mathcal{A}(\Sigma_P) \ar[r]^{\phi} \ar[d] & \operatorname{Gr}^{\mathcal{N}} A_{P}(0) \ar[d]^{\mathbf{t}^u \cdot }\\
\mathcal{A}\left(\boldsymbol{\Sigma}_P(u)\right) \ar[r]^{\phi_u} & \operatorname{Gr}^{\mathcal{N}} A_{P}(u)
}\end{align}

\item
$\phi\left(\chi_{\rho}\right) = a_{\rho} \mathbf{t}^{v_{\rho}} \cdot \mathbf{t}^{u}$ for all $\rho \in \Sigma_P(u)(1)$.
\end{enumerate}
\end{lemma}
\begin{proof}
Consider $\operatorname{Gr}^{\mathcal{N}} K\left[\mathbf{t}^{ \pm 1}\right]$.
As $\deg_P$ is strictly convex, we have $\mathcal{N}_{\alpha_1} \cdot \mathcal{N}_{\alpha_2} \subset \mathcal{N}_{\alpha_1 + \alpha_2}$.
Hence $\operatorname{Gr}^{\mathcal{N}} K\left[\mathbf{t}^{ \pm 1}\right]$ has a graded $K$-algebra structure.
We have
\begin{align}\label{align:mult-gr-jaco}
\mathbf{t}^{u_1} \cdots \mathbf{t}^{u_k} = \begin{cases}
\mathbf{t}^{u_1 + \cdots + u_k}, & u_1, \cdots,  u_k \text{ are cofacial},\\
0, & \text{otherwise},
\end{cases}
\end{align}
in $\operatorname{Gr}^{\mathcal{N}} K\left[\mathbf{t}^{ \pm 1}\right]$.
Consider the map
$$\begin{aligned}
\tilde{\phi}_u : K[x_{\rho}]_{\rho \in \boldsymbol{\Sigma}_P(u)(1)} &\to \operatorname{Gr}^{\mathcal{N}} K[\mathbf{t}^{ \pm 1}]\\
\prod_{\rho} x_{\rho}^{n_{\rho}} &\mapsto \prod_{\rho} \left(a_{\rho} \mathbf{t}^{v_{\rho}}\right)^{n_{\rho}} \cdot \mathbf{t}^{u}
\end{aligned}$$
For any monomial $x_{\rho_1}^{n_1} \cdots x_{\rho_k}^{n_k}$, where $n_i \geq 1$,
by (\ref{align:mult-gr-jaco}), 
we know that $\tilde{\phi}_u\left(x_{\rho_1}^{n_1} \cdots x_{\rho_k}^{n_k}\right) \neq 0$ if and only if $\rho_1, \cdots, \rho_k$ are cofacial.
Therefore, $\tilde{\phi}_u$ factor through the Stanley-Reisner ring  $\operatorname{SR}\left[\boldsymbol{\Sigma}_P(u)\right]$ and we get $\phi_u : \operatorname{SR}\left[\boldsymbol{\Sigma}_P(u)\right] \to \operatorname{Gr}^{\mathcal{N}} K[\mathbf{t}^{ \pm 1}]$.

Notice that for any element in $P_u(\boldsymbol{\Sigma})$, there exists one and only one way to write it in the form 
$u + n_{1} v_{\rho_1} + \cdots + n_{r} v_{\rho_r}$, 
where $\rho_1, \cdots, \rho_r \in \boldsymbol{\Sigma}_P(u)(1)$ are cofacial and $n_i \in \mathbb{Z}_{> 0}$.
Therefore, 
$\phi_u$ is injective and $\operatorname{im} \phi_u = \operatorname{Gr}^{\mathcal{N}} A_{P}(u)$.
We then use the fact that $\operatorname{SR}\left[\boldsymbol{\Sigma}_P(u)\right] \cong \mathcal{A}\left(\boldsymbol{\Sigma}_P(u)\right)$ to get an isomorphism $\phi_u : \mathcal{A}\left(\boldsymbol{\Sigma}_P(u)\right) \to \operatorname{Gr}^{\mathcal{N}} A_{P}(u)$.
\end{proof}

Let
$$\begin{aligned}
\overline{\Omega}^{k}_{\mathcal{N}}(u) 
&\defeq
\bigoplus_{i_1 < \cdots < i_k} \operatorname{Gr}^{\mathcal{N}}A_P(u) \cdot \frac{\mathrm{d} t_{i_1}}{t_{i_1}} \wedge \cdots \wedge \frac{\mathrm{d} t_{i_k}}{t_{i_k}} \subset \operatorname{Gr}^{\mathcal{N}}\Omega^{\bullet}(U),\\
\Omega^{k}_{\mathcal{N}}(u) 
&\defeq
\bigoplus_i \bigoplus_{i_1 < \cdots < i_k} \tau^i \operatorname{Gr}^{\mathcal{N}}_{\{\deg(u)\}+i+k-n} A_P(u) \cdot \frac{\mathrm{d} t_{i_1}}{t_{i_1}} \wedge \cdots \wedge \frac{\mathrm{d} t_{i_k}}{t_{i_k}}\\
&\subset \operatorname{Gr}^{\mathcal{N}}\Omega^{\bullet}(U)\left[\tau^{\pm 1}\right],
\end{aligned}$$
for any $u \in \operatorname{Box}(\boldsymbol{\Sigma}_P)$.
Then we have 
$$\begin{aligned}
\operatorname{Gr}^{\mathcal{N}}\Omega^{\bullet}(U) &= 
\bigoplus_{u \in \operatorname{Box}(\boldsymbol{\Sigma}_P)} \overline{\Omega}^{k}_{\mathcal{N}}(u),\\
\operatorname{Gr}^{\mathcal{N}}_{\alpha}\Omega^{\bullet}(U)\left[\tau^{\pm 1}\right] &= 
\bigoplus_{\substack{u \in \operatorname{Box}(\boldsymbol{\Sigma}_P)\\\{\deg(u)\} = \alpha}} \Omega^{k}_{\mathcal{N}}(u) , \quad 0 \leq \alpha < 1.
\end{aligned}$$
Note that the operator $\theta \mathrm{d} - \mathrm{d} f \wedge $ (resp. $- \mathrm{d} f \wedge $) preserves the above decomposition. 
So we have well-defined complexes
$$
K_{\mathcal{N}}(u) = \left(\overline{\Omega}^{\bullet}_{\mathcal{N}}(u), - \mathrm{d} f \wedge \right) \subset \operatorname{Gr}^{\mathcal{N}}K(f)
$$ 
and
$$
\Omega_{\mathcal{N}}(u) = \left(\Omega^{\bullet}_{\mathcal{N}}(u), \theta \mathrm{d} - \mathrm{d} f \wedge \right) \subset \operatorname{Gr}^{\mathcal{N}}\Omega(f),
$$
Let
$$\begin{aligned}
J_f^{\mathcal{N}}(u) &= H^n\left(K_{\mathcal{N}}(u)\right),\\
H(u) &= H^n\left(\Omega_{\mathcal{N}}(u)\right).
\end{aligned}$$
We have the linear transformation $f$ on $J_f^{\mathcal{N}}(u)$ and the linear transformation $N$ on $H(u)$.
We have
\begin{align}
\label{align:decomposition-H}
H_{\alpha} &= \bigoplus_{\substack{u \in \operatorname{Box}(\boldsymbol{\Sigma}_P)\\\{\deg(u)\} = \alpha}} H(u), \quad 0 \leq \alpha < 1\\
\label{align:decomposition-jac}
\operatorname{Gr}^{\mathcal{N}} J_{f} &= \bigoplus_{u \in \operatorname{Box}(\boldsymbol{\Sigma}_P)} J_{f}^{\mathcal{N}}(u).
\end{align}

\begin{proposition}
\label{prop:jac-SR-rel}
$J^{\mathcal{N}}_{P}(u)$ is a free $H(\Sigma_P(u))$-module with basis $\mathbf{t}^u$. (For the definition of $H(\Sigma_P(u))$, see Definition \ref{def:conewise_polynomial}.)
The action of $f \in J^{\mathcal{N}}_{P}$ on $J^{\mathcal{N}}_{P}(u)$ corresponds to the action of $\deg_P \in H(\Sigma_P)$ on $H(\Sigma_P(u))$. (For the definition of $\deg_P$, see (\ref{def_deg}).)
\end{proposition}
\begin{proof}
For any
$m 
\in \left(N_{\mathbb{Q}}\right)^{\vee} 
= \mathfrak{m}_1 \subset K[\mathbf{t}]$, we have
$$
m = \sum_{\rho \in \boldsymbol{\Sigma}_P(1)} m(v_{\rho}) \chi_{\rho}
$$
in $\mathcal{A}(\Sigma_P)$.
Hence 
$$
\phi(m) 
= \phi\left(\sum_{\rho \in \boldsymbol{\Sigma}_P(1)} m(v_{\rho}) \chi_{\rho}\right) 
= \sum_{v \in P(0)} m(v) a_v t^v.
$$
Therefore
$$\begin{gathered}
\phi(t_i) 
= \sum_{v \in P(0)} v_i a_v t^v 
= t_i\frac{\partial f}{\partial t_i},\\
\phi(\deg_P) 
= \sum_v a_v t^v 
= f.
\end{gathered}$$
Hence (\ref{align:fin}) induces an isomorphism
$$\begin{aligned}
&H\left(\Sigma_P(u)\right) 
\cong \mathcal{A}\left(\Sigma_P(u)\right)/\left(t_1, \cdots, t_n\right)\mathcal{A}\left(\Sigma_P(u)\right) \\
\rightarrow
&J^{\mathcal{N}}_{P}(u) 
\cong \left.\operatorname{Gr}^{\mathcal{N}}A_{P}(u)\middle/\left(t_1\frac{\partial f}{\partial t_1}, \cdots, t_n\frac{\partial f}{\partial t_n}\right)\operatorname{Gr}^{\mathcal{N}}A_{P}(u).\right.
\end{aligned}$$
And the action of $f$ corresponds to the action of $\deg_P$.
\end{proof}

\begin{corollary}
\label{lemma:van-jac-rel-u}
Let $K=\mathbb{C}$. For any $u \in \operatorname{Box}(\boldsymbol{\Sigma}_P)$,
\begin{enumerate}
\item 
we have a (non-canonical) isomorphism
$$
\left(H(u), N\right) 
\xrightarrow{\sim}
\left(\operatorname{Gr}^{\Phi} H(u), N\right),
$$

\item
we have canonical isomorphisms
$$
\left(\operatorname{Gr}^{\Phi} H(u), N\right) 
\xrightarrow{\sim}
\left(J_f^{\mathcal{N}}(u), [f]\right)
\xrightarrow{\sim}
\left(H\left(\Sigma_P(u)\right), \deg_P\right).
$$
Under these isomorphisms we have
$$
\operatorname{Gr}^{\Phi}_p H(u)
\xrightarrow{\sim}
\left(J_f^{\mathcal{N}}(u)\right)_{p + \{\deg(u)\}}
\xrightarrow{\sim}
H^{p - \lfloor\deg(u)\rfloor}\left(\Sigma_P(u)\right).
$$
\end{enumerate}    
\end{corollary}
\begin{proof}
\begin{enumerate}
\item 
Note that for homomophisms of filtered modules $f : (A_1, F) \to (B_1, F)$, $g : (A_2, F) \to (B_2, F)$, $f$ and $g$ are strict if and only if $f \oplus g$ is strict. 
Hence by the fact that $N^j : (H, G_{\bullet}) \to (H, G_{\bullet})[j]$ is strict, we know that $N^j : (H(u), G_{\bullet}) \to (H(u), G_{\bullet})[j]$ is strict for any $j$.
Therefore, by Lemma \ref{lemma:van-jac-rel}, we have a non-canonical isomorphism
$$
\left(H(u), N\right) 
\xrightarrow{\sim}
\left(\operatorname{Gr}^{\Phi} H(u), N\right).
$$

\item
By the fact that $H^i\left(\operatorname{Gr}^{\mathcal{N}} K(f)\right) = 0$, we know that $H^i\left(K_{\mathcal{N}}(u)\right) = 0$ for all $i \neq n$.
Hence, by the same proof of (\ref{align:gr-G-on-grade-N-G}) and (\ref{align:global-H-J-equal}),
the spectral sequence associated to $\left(\Omega_{\mathcal{N}}(u), {\Phi}_{\bullet}\right)$ gives an isomorphism
$$
\left(\operatorname{Gr}^{\Phi} H(u), N\right) 
\xrightarrow{\sim}
\left(J_f^{\mathcal{N}}(u), [f]\right).
$$
By Proposition \ref{prop:jac-SR-rel}, we have $H(\Sigma_P(u)) \xrightarrow{\sim} J^{\mathcal{N}}_{P}(u)$.
\end{enumerate}
\end{proof}

For any $u \in \operatorname{Box}(\boldsymbol{\Sigma})$, denote $\sigma = \sigma(u)$ and 
$u^{-1} = \sum_{\rho \in \sigma(1)} v_{\rho} - u$. 
Notice that $\sigma(u^{-1}) = \sigma(u) = \sigma$ and $\deg(u) + \deg(u^{-1}) = \dim \sigma$.
We have $\lfloor\deg(u)\rfloor + \lfloor\deg(u^{-1})\rfloor = \dim \sigma + \nu - n$, where $\nu$ is defined in Definition \ref{def:vanishing-cycle}.
Note that $\operatorname{Box}(\sigma)$ is in one to one correspondence with $N(\sigma)$ and if we view $u$ and $u^{-1}$ as elements in $N(\sigma)$, then they are inverse to each other.

\begin{corollary}
\label{cor:Birkhoff problem}
Let $K=\mathbb{C}$. 
\begin{enumerate}[(1)]
\item 
We have 
$$
(H, F^{\bullet}, N) \cong \bigoplus_{u \in \operatorname{Box}(\boldsymbol{\Sigma}_P)} \left(H\left(\Sigma_P(u)\right), F^{\bullet - \lfloor\deg(u^{-1})\rfloor}, \deg_P\right).
$$

\item
$F^{\bullet + \lfloor\deg(u^{-1})\rfloor}$ and $M(N)_{2 \bullet - \operatorname{codim} \sigma}$ are opposite filtrations on $H(u)$, i.e.
$$\begin{aligned}
H(u) &= M(N)_{2k - \operatorname{codim} \sigma} \oplus F^{k + \lfloor\deg(u^{-1})\rfloor + 1} H(u)
\end{aligned}$$
for all $k$.
\end{enumerate}
\end{corollary}
\begin{proof}
\begin{enumerate}[(1)]
\item 
By Corollary \ref{lemma:van-jac-rel-u}, we have an isomorphism
$$
\left(H(u), N\right) 
\xrightarrow{\sim}
\left(H\left(\Sigma_P(u)\right), \deg_P\right),
$$
such that the image of $F^kH(u)$ is
$$
\bigoplus_{p \leq \nu - k} H^{p - \lfloor\deg(u)\rfloor}\left(\Sigma_P(u)\right)
= F^{k - \lfloor\deg(u^{-1})\rfloor}H\left(\Sigma_P(u)\right)
$$

\item
By Corollary \ref{lemma:star-quotient-isom}, we have
$$
H\left(\Sigma_P(u)\right) = M(\deg_P)_{2k-\operatorname{codim} \sigma} 
\oplus F^{k+1} H\left(\Sigma_P(u)\right).
$$
Therefore we know that 
$$\begin{aligned}
H\left(u\right) 
&= M(N)_{2k - \operatorname{codim} \sigma} \oplus F^{k + 1 + \lfloor\deg(u^{-1})\rfloor} H(u).
\end{aligned}$$
\end{enumerate}
\end{proof}

\begin{definition}
Let $A_{pq} = (A, F_{p,q}^{\bullet}, (-1)^q Q)$ be the following polarized Hodge structure of weight $p+q$, ($p,q \in \mathbb{Z}$):
\begin{itemize}
\item 
$A = \mathbb{Z} e_1 \oplus \mathbb{Z} e_2$ is the free $\mathbb{Z}$-module of rank $2$.

\item
$F_{p,q}^{\bullet} \defeq F_p^{\bullet}(\mathbb{C} z) \oplus \overline{F_q^{\bullet}(\mathbb{C} z)}$ is a decreasing filtration on 
$$
A_{\mathbb{C}} = \mathbb{C} e_1 \oplus \mathbb{C} e_2 = \mathbb{C} z \oplus \mathbb{C} \overline{z}, \quad z = e_1 + \mathrm{i} e_2,
$$
where
$$
F_p^k(\mathbb{C} z) = \begin{cases}
\mathbb{C} z, & k \leq p,\\
0, & p+1 \leq k,
\end{cases}
$$
for any $p \in \mathbb{Z}$.

\item
$Q$ is the bilinear form on $A_{\mathbb{Q}}$ such that the matrix of $Q$ with respect to the basis $\{e_1, e_2\}$ is $\begin{pmatrix}
1 \\
& 1
\end{pmatrix}$ if $p+q$ is even, and is 
$\begin{pmatrix}
& 1 \\
-1
\end{pmatrix}$
if $p+q$ is odd.
\end{itemize}
\end{definition}

\begin{proposition}\label{prop:PMHS-on-J}
Let $K = \mathbb{C}$.
For $u \in \operatorname{Box}(\boldsymbol{\Sigma}_P)$, 
\begin{enumerate}[(a)]
\item 
let 
$$
\nu = \begin{cases}
n, & \deg(u) \in \mathbb{Z},\\
n-1, & \deg(u) \notin \mathbb{Z},\\
\end{cases}
$$

\item 
let $F^{p} = \mathcal{N}_{n - p} = \oplus_{i \leq \nu - p} \left(J^{\mathcal{N}}_{f}(u)\right)_{i + \{\deg(u)\}}$ be the decreasing filtration on $J^{\mathcal{N}}_{f}(u)$ induced by the Newton filtration,

\item
and let $W_{\bullet} = M(f)_{\bullet - \nu}$, where $M(f)_{\bullet}$ is the monodromy filtration of $f$  on $J^{\mathcal{N}}_{f}(u)$.
\end{enumerate}
Let $\sigma = \sigma(u)$ and let $\boldsymbol{\Sigma}_P(\sigma)$ be defined as in Definition \ref{def:quotient-stacky-fan}(c).
Then
\begin{enumerate}[(1)]
\item
The isomorphism $J^{\mathcal{N}}_{f}(u) \cong H(u)$ in Corollary \ref{lemma:van-jac-rel-u} is compatible with $F^{\bullet}$ and $W_{\bullet}$.

\item 
\begin{enumerate}
\item 
if $u = u^{-1}$, 
then $\left(J^{\mathcal{N}}_{f}(u), F^{\bullet}, W_{\bullet}\right)$ underlies a polarized mixed Hodge structure with weight $\nu$ which is isomorphic to the polarized mixed Hodge structure on $H(\boldsymbol{\Sigma}_P(\sigma))\left(\lfloor \deg(u) \rfloor\right)$, where $\mathfrak{A}(k)$ is the $k$-th Tate twist of a polarized mixed Hodge structure $\mathfrak{A}$.

\item
if $u \neq u^{-1}$, 
then $\left(J^{\mathcal{N}}_{f}(u) \oplus J^{\mathcal{N}}_{f}(u^{-1}), F^{\bullet}, W_{\bullet}\right)$  underlies a polarized mixed Hodge structure  with weight $\nu$  which is isomorphic to the polarized mixed Hodge structure on $H(\boldsymbol{\Sigma}_P(\sigma)) \otimes A_{\lfloor \deg(u^{-1}) \rfloor, \lfloor \deg(u) \rfloor}$.
\end{enumerate}
\end{enumerate}
\end{proposition}
\begin{proof}
\begin{enumerate}[(1)]
\item
Under the isomorphisms in Corollary \ref{lemma:van-jac-rel-u}, we have
$$
F^p H(u)
\xrightarrow{\sim} 
\oplus_{i \leq \nu - p} \operatorname{Gr}_p^{\Phi} H(u) 
\xrightarrow{\sim} 
\oplus_{i \leq \nu - p} \left(J^{\mathcal{N}}_{f}(u)\right)_{i + \{\deg(u)\}} 
= 
F^{p} J^{\mathcal{N}}_{f}(u)
$$
and $W(N)_{\bullet} \xrightarrow{\sim} W(f)_{\bullet}$.

\item 
By Corollary \ref{lemma:van-jac-rel-u},
we have an isomorphism $\psi_u : J^{\mathcal{N}}_{f}(u) \xrightarrow{\sim} H_{\mathbb{C}}(\boldsymbol{\Sigma}_P(\sigma))$, such that 
$$\psi_u\left(J^{\mathcal{N}}_{f}(u)_{i + \deg(u)} \right) = H_{\mathbb{C}}^i(\boldsymbol{\Sigma}_P(\sigma)), 
\quad \psi_u\left(M(f)_i\right) = M(\deg_P)_i$$ 
for all $i$.
Hence
$$\begin{aligned}
\psi_u\left(F^pJ^{\mathcal{N}}_{f}(u)\right) 
&= \bigoplus_{i \leq \nu - p} \psi_u\left(\left(J^{\mathcal{N}}_{f}(u)\right)_{i + \{\deg(u)\}}\right)\\
&= \bigoplus_{i \leq \nu - p} H_{\mathbb{C}}^{i-\lfloor \deg(u) \rfloor}(\boldsymbol{\Sigma}_P(\sigma))\\
&= F^{p - \lfloor \deg(u^{-1}) \rfloor}H_{\mathbb{C}}(\boldsymbol{\Sigma}_P(\sigma))\\
\psi_u\left(W_pJ^{\mathcal{N}}_{f}(u)\right)
&= \psi_u\left(M(f)_{p - \nu}\right)
= M(\deg_P)_{p - \nu}
= W_{p - \tilde{\nu}}H_{\mathbb{C}}(\boldsymbol{\Sigma}_P(\sigma)).
\end{aligned}$$
where $\tilde{\nu} \defeq \dim \sigma + \nu - n = \lfloor \deg(u) \rfloor + \lfloor \deg(u^{-1}) \rfloor$.

\begin{enumerate}
\item 
In this case, we have $\deg(u^{-1}) = \deg(u)$. Hence
$$\begin{aligned}
\psi_u\left(F^pJ^{\mathcal{N}}_{f}(u)\right) 
&= \left(F^{\bullet}H_{\mathbb{C}}(\boldsymbol{\Sigma}_P(\sigma))\right)^{p-\lfloor \deg(u) \rfloor}\\
\psi_u\left(W_pJ^{\mathcal{N}}_{f}(u)\right)
&= \left(W_{\bullet}H_{\mathbb{C}}(\boldsymbol{\Sigma}_P(\sigma))\right)_{p - 2\lfloor \deg(u) \rfloor}
\end{aligned}$$

\item
Consider the isomorphism
$$\begin{aligned}
\psi_{u} : J^{\mathcal{N}}_{f}(u) \oplus J^{\mathcal{N}}_{f}(u^{-1}) &\to H_{\mathbb{C}}(\boldsymbol{\Sigma}_P(\sigma)) \otimes A_{\lfloor \deg(u) \rfloor, \lfloor \deg(u^{-1}) \rfloor}\\
(f, g) &\mapsto \varphi_u(f) \otimes z + \varphi_{u^{-1}}(f) \otimes \bar{z}.
\end{aligned}
$$
Then
$$\begin{aligned}
&\psi_u\left(F^pJ^{\mathcal{N}}_{f}(u) \oplus F^pJ^{\mathcal{N}}_{f}(u^{-1})\right) \\
= &F^{p - \lfloor \deg(u^{-1}) \rfloor}H_{\mathbb{C}}(\boldsymbol{\Sigma}_P(\sigma))
 \otimes z + F^{p - \lfloor \deg(u) \rfloor}H_{\mathbb{C}}(\boldsymbol{\Sigma}_P(\sigma)) \otimes \bar{z}\\
= &\sum_{i + j = p} \left(F^{i}H_{\mathbb{C}}(\boldsymbol{\Sigma}_P(\sigma)) \otimes F_{\lfloor \deg(u^{-1}) \rfloor}^j \mathbb{C} z + F^{i}H_{\mathbb{C}}(\boldsymbol{\Sigma}_P(\sigma)) \otimes \overline{F_{\lfloor \deg(u) \rfloor}^j \mathbb{C} z}\right)\\
= &\sum_{i + j = p} F^{i}H_{\mathbb{C}}(\boldsymbol{\Sigma}_P(\sigma)) \otimes F^j A_{\lfloor \deg(u^{-1}) \rfloor, \lfloor \deg(u) \rfloor},\\
&\psi_u\left(W_pJ^{\mathcal{N}}_{f}(u) \oplus W_pJ^{\mathcal{N}}_{f}(u^{-1})\right)\\
= &W_{p - \tilde{\nu}}H_{\mathbb{C}}(\boldsymbol{\Sigma}_P(\sigma)) \otimes z + W_{p - \tilde{\nu}}H_{\mathbb{C}}(\boldsymbol{\Sigma}_P(\sigma)) \otimes \bar{z}\\
= &\sum_{i + j = p} W_{i}H_{\mathbb{C}}(\boldsymbol{\Sigma}_P(\sigma)) \otimes W_j A_{\lfloor \deg(u^{-1}) \rfloor, \lfloor \deg(u) \rfloor}.
\end{aligned}$$
\end{enumerate}
\end{enumerate}
\end{proof}

\begin{remark}
By Proposition \ref{prop:PMHS-on-J}, we can construct two polarized mixed Hodge structure  with weight $n$ ($n-1$ respectively) on $\oplus_{\alpha \in \mathbb{Z}} \operatorname{Gr}^{\mathcal{N}}_{\alpha} J_{f}$ ($\oplus_{\alpha \notin \mathbb{Z}} \operatorname{Gr}^{\mathcal{N}}_{\alpha} J_{f}$ respectively):
\begin{enumerate}
\item By Proposition \ref{prop:PMHS-on-J} (1), we can use the isomorphism
$$
\operatorname{Gr}^{\mathcal{N}} J_{f} \cong \bigoplus_{u \in \operatorname{Box}(\boldsymbol{\Sigma}_P)} J_{f}^{\mathcal{N}}(u) \cong \bigoplus_{u \in \operatorname{Box}(\boldsymbol{\Sigma}_P)} H(u) \cong H
$$
to construct such a structure.

\item
By Proposition \ref{prop:PMHS-on-J} (2), we can also use the isomorphism
$$\begin{aligned}
&\operatorname{Gr}^{\mathcal{N}} J_{f} \\
\cong &\bigoplus_{u \in \operatorname{Box}(\boldsymbol{\Sigma}_P)} J_{f}^{\mathcal{N}}(u) \\
= &\bigoplus_{\substack{u \in \operatorname{Box}(\boldsymbol{\Sigma}_P)\\u = u^{-1}}} J_{f}^{\mathcal{N}}(u) \oplus \bigoplus_{\substack{\{u, u^{-1}\} \subset \operatorname{Box}(\boldsymbol{\Sigma}_P)\\u \neq u^{-1}}} \left(J_{f}^{\mathcal{N}}(u) \oplus J^{\mathcal{N}}_{f}(u^{-1})\right)\\
\cong &
\bigoplus_{\substack{u \in \operatorname{Box}(\boldsymbol{\Sigma}_P)\\u = u^{-1}}} H(u)(\lfloor\deg(u)\rfloor) \oplus \bigoplus_{\substack{\{u, u^{-1}\} \subset \operatorname{Box}(\boldsymbol{\Sigma}_P)\\u \neq u^{-1}}} \left(H(\boldsymbol{\Sigma}_P(\sigma(u))) \otimes A_{\lfloor \deg(u^{-1}) \rfloor, \lfloor \deg(u) \rfloor}\right)
\end{aligned}$$
to construct such a structure.
\end{enumerate}
We know that they have the same Hodge filtration and weight filtration, hence the same Hodge diamond.
By we do not know whether they have the same $\mathbb{Q}$-structure.
\end{remark}

\begin{definition}
\label{def:diamond}
\begin{enumerate}[(a)]
\item 
For any $m \times n$-matrix 
$$
\boldsymbol{A} = \begin{pmatrix}
a_{11} & \cdots & a_{1n}\\
\vdots & \ddots & \vdots\\
a_{m1} & \cdots & a_{mn}\\
\end{pmatrix}
$$
define a Hodge diamond $\operatorname{HD}(\boldsymbol{A})$ of weight $m + n$ to be
$$
\operatorname{HD}(\boldsymbol{A}) \defeq
0
\begin{array}{c}
\iddots\\
\\
\ddots
\end{array}
\begin{array}{c}
0\\
\\
\\
\\
0
\end{array}
\begin{array}{c}
0\\
\\
\\
\\
\\
\\
0
\end{array}
\begin{array}{c}
a_{11}\\
\\
a_{21}\\
\\
\vdots\\
\\
a_{m-1,1}\\
\\
a_{m1}
\end{array}
\begin{array}{c}
0\\
\\
0\\
\\
\\
\\
\\
\\
0\\
\\
0
\end{array}
\begin{array}{c}
\iddots\\
\\
\\
a_{12}\\
\\
a_{22}\\
\\
\vdots\\
\\
a_{m-1,2}\\
\\
a_{m2}\\
\\
\\
\ddots\\
\end{array}
\begin{array}{c}
0\\
\\
\\
\\
\\
\cdots\\
\\
\cdots\\
\\
\cdots\\
\\
\cdots\\
\\
\cdots\\
\\
\\
\\
\\
0
\end{array}
\begin{array}{c}
\ddots\\
\\
\\
a_{1,n-1}\\
\\
a_{2,n-1}\\
\\
\vdots\\
\\
a_{m-1,n-1}\\
\\
a_{m,n-1}\\
\\
\\
\iddots\\
\end{array}
\begin{array}{c}
0\\
\\
0\\
\\
\\
\\
\\
\\
0\\
\\
0
\end{array}
\begin{array}{c}
a_{1n}\\
\\
a_{2n}\\
\\
\vdots\\
\\
a_{m-1,n}\\
\\
a_{mn}
\end{array}
\begin{array}{c}
0\\
\\
\\
\\
\\
\\
0
\end{array}
\begin{array}{c}
0\\
\\
\\
\\
0
\end{array}
\begin{array}{c}
\ddots\\
\\
\iddots
\end{array}
0
$$

\item
For any face $\sigma \prec P$, let
$$
n(\sigma, \alpha) \defeq |\{u \in \operatorname{Box}(\sigma), \deg_P(u) = \alpha\}|. 
$$
Let $\left(h_0(\Sigma_P(\sigma)), h_1(\Sigma_P(\sigma)), \cdots, h_{\operatorname{codim} \sigma}(\Sigma_P(\sigma))\right)$ be the $h$-vector defined in Definition \ref{def:h-vector} and let $\boldsymbol{A}_{\alpha}(\sigma)$ be the matrix
$$
\begin{pmatrix}
h_0(\Sigma_P(\sigma))\\
h_1(\Sigma_P(\sigma))\\
\vdots\\
h_{\operatorname{codim} \sigma}(\Sigma_P(\sigma))\\
\end{pmatrix}
\begin{pmatrix}
n(\sigma, 0) & n(\sigma, 1) & \cdots & n(\sigma, n)
\end{pmatrix}
$$
when $\alpha = 0$, and 
$$
\begin{pmatrix}
h_0(\Sigma_P(\sigma))\\
h_1(\Sigma_P(\sigma))\\
\vdots\\
h_{\operatorname{codim} \sigma}(\Sigma_P(\sigma))\\
\end{pmatrix}
\begin{pmatrix}
n(\sigma, \alpha) & n(\sigma, \alpha + 1) & \cdots & n(\sigma, \alpha + n - 1).
\end{pmatrix}
$$
when $0<\alpha<1$.
Let
$$\begin{aligned}
&\operatorname{HD}_{\alpha}(\sigma) \defeq \operatorname{HD}(\boldsymbol{A}_{\alpha}(\sigma)),
&&\operatorname{HD}_{\neq 0}(\sigma) \defeq \sum_{0<\alpha<1} \operatorname{HD}_{\alpha}(\sigma),\\
&\operatorname{HD}_0 \defeq \sum_{\sigma} \operatorname{HD}_0(\sigma),
&&\operatorname{HD}_{\neq 0} \defeq \sum_{\sigma} \operatorname{HD}_{\neq 0}(\sigma).
\end{aligned}$$
\end{enumerate}
\end{definition}

\begin{corollary}\label{cor:PMHS-van-cyc}
\begin{enumerate}
\item 
The Hodge diamonds of both $H_0$ and $\oplus_{\alpha \in \mathbb{Z}} \operatorname{Gr}^{\mathcal{N}}_{\alpha} J_{f}$ are $\operatorname{HD}_{0}$.
\item 
The Hodge diamonds of both $H_{\neq 0}$ and $\oplus_{\alpha \notin \mathbb{Z}} \operatorname{Gr}^{\mathcal{N}}_{\alpha} J_{f}$ are $\operatorname{HD}_{\neq 0}$.
\end{enumerate}
\end{corollary}

\begin{remark}
\label{rmk:sub-diagram deformation}
For any sub-diagram deformation $f^{\prime}$ of $f$, we have an isomorphism $\operatorname{Gr}^{\mathcal{N}} J_{f^{\prime}} \cong \operatorname{Gr}^{\mathcal{N}} J_{f}$ and the vanishing cycles of $f$ and $f'$ are also isomorphic to each other.
Hence Corollary \ref{cor:PMHS-van-cyc} holds for any sub-diagram deformation of $f_{P, \mathbf{a}}$.

Moreover, for general non-degenerate $f$, $\dim \operatorname{Gr}^{W}_pH$ and $\dim \operatorname{Gr}_{F}^pH$ depends only on $P$ for any $p$.  (See \cite{sabbah2018some}.)
Hence we can use Corollary \ref{cor:PMHS-van-cyc} to compute them.
\end{remark}

\backmatter



\bmhead{Acknowledgements}

I want to thank Professor Lei Fu for guiding and supporting me throughout this study. 
His expertise and encouragement were crucial to its successful completion.

\section*{Declarations}

The corresponding author states that there is no conflict of interest.
All data included in this study are available upon request by contact with the corresponding author.


\bibliography{refs}


\begin{thebibliography}{23}
\ifx \bisbn   \undefined \def \bisbn  #1{ISBN #1}\fi
\ifx \binits  \undefined \def \binits#1{#1}\fi
\ifx \bauthor  \undefined \def \bauthor#1{#1}\fi
\ifx \batitle  \undefined \def \batitle#1{#1}\fi
\ifx \bjtitle  \undefined \def \bjtitle#1{#1}\fi
\ifx \bvolume  \undefined \def \bvolume#1{\textbf{#1}}\fi
\ifx \byear  \undefined \def \byear#1{#1}\fi
\ifx \bissue  \undefined \def \bissue#1{#1}\fi
\ifx \bfpage  \undefined \def \bfpage#1{#1}\fi
\ifx \blpage  \undefined \def \blpage #1{#1}\fi
\ifx \burl  \undefined \def \burl#1{\textsf{#1}}\fi
\ifx \doiurl  \undefined \def \doiurl#1{\url{https://doi.org/#1}}\fi
\ifx \betal  \undefined \def \betal{\textit{et al.}}\fi
\ifx \binstitute  \undefined \def \binstitute#1{#1}\fi
\ifx \binstitutionaled  \undefined \def \binstitutionaled#1{#1}\fi
\ifx \bctitle  \undefined \def \bctitle#1{#1}\fi
\ifx \beditor  \undefined \def \beditor#1{#1}\fi
\ifx \bpublisher  \undefined \def \bpublisher#1{#1}\fi
\ifx \bbtitle  \undefined \def \bbtitle#1{#1}\fi
\ifx \bedition  \undefined \def \bedition#1{#1}\fi
\ifx \bseriesno  \undefined \def \bseriesno#1{#1}\fi
\ifx \blocation  \undefined \def \blocation#1{#1}\fi
\ifx \bsertitle  \undefined \def \bsertitle#1{#1}\fi
\ifx \bsnm \undefined \def \bsnm#1{#1}\fi
\ifx \bsuffix \undefined \def \bsuffix#1{#1}\fi
\ifx \bparticle \undefined \def \bparticle#1{#1}\fi
\ifx \barticle \undefined \def \barticle#1{#1}\fi
\bibcommenthead
\ifx \bconfdate \undefined \def \bconfdate #1{#1}\fi
\ifx \botherref \undefined \def \botherref #1{#1}\fi
\ifx \url \undefined \def \url#1{\textsf{#1}}\fi
\ifx \bchapter \undefined \def \bchapter#1{#1}\fi
\ifx \bbook \undefined \def \bbook#1{#1}\fi
\ifx \bcomment \undefined \def \bcomment#1{#1}\fi
\ifx \oauthor \undefined \def \oauthor#1{#1}\fi
\ifx \citeauthoryear \undefined \def \citeauthoryear#1{#1}\fi
\ifx \endbibitem  \undefined \def \endbibitem {}\fi
\ifx \bconflocation  \undefined \def \bconflocation#1{#1}\fi
\ifx \arxivurl  \undefined \def \arxivurl#1{\textsf{#1}}\fi
\csname PreBibitemsHook\endcsname

\bibitem[\protect\citeauthoryear{Sabbah}{1999}]{sabbah1999hypergeometric}
\begin{barticle}
\bauthor{\bsnm{Sabbah}, \binits{C.}}:
\batitle{{Hypergeometric period for a tame polynomial}}.
\bjtitle{Comptes Rendus de l'Acad{\'e}mie des Sciences-Series I-Mathematics}
\bvolume{328}(\bissue{7}),
\bfpage{603}--\blpage{608}
(\byear{1999})
\doiurl{10.1016/S0764-4442(99)80254-7}
\end{barticle}
\endbibitem

\bibitem[\protect\citeauthoryear{Douai and Sabbah}{2003}]{douai2003gauss}
\begin{barticle}
\bauthor{\bsnm{Douai}, \binits{A.}},
\bauthor{\bsnm{Sabbah}, \binits{C.}}:
\batitle{{Gauss-Manin systems, Brieskorn lattices and Frobenius structures (I)}}.
\bjtitle{Annales de l'institut Fourier}
\bvolume{53}(\bissue{4}),
\bfpage{1055}--\blpage{1116}
(\byear{2003})
\doiurl{10.5802/aif.1974}
\end{barticle}
\endbibitem

\bibitem[\protect\citeauthoryear{Steenbrink and Zucker}{1985}]{steenbrink1985variation}
\begin{barticle}
\bauthor{\bsnm{Steenbrink}, \binits{J.}},
\bauthor{\bsnm{Zucker}, \binits{S.}}:
\batitle{{Variation of mixed Hodge structure. I}}.
\bjtitle{Inventiones mathematicae}
\bvolume{80}(\bissue{3}),
\bfpage{489}--\blpage{542}
(\byear{1985})
\doiurl{10.1007/BF01388729}
\end{barticle}
\endbibitem

\bibitem[\protect\citeauthoryear{Saito}{1989}]{saito1989structure}
\begin{barticle}
\bauthor{\bsnm{Saito}, \binits{M.}}:
\batitle{{On the structure of Brieskorn lattice}}.
\bjtitle{Annales de l'institut Fourier}
\bvolume{39}(\bissue{1}),
\bfpage{27}--\blpage{72}
(\byear{1989})
\doiurl{10.5802/aif.1157}
\end{barticle}
\endbibitem

\bibitem[\protect\citeauthoryear{Sabbah}{1997}]{sabbah1997monodromy}
\begin{barticle}
\bauthor{\bsnm{Sabbah}, \binits{C.}}:
\batitle{{Monodromy at infinity and Fourier transform}}.
\bjtitle{Publications of the Research Institute for Mathematical Sciences}
\bvolume{33}(\bissue{4}),
\bfpage{643}--\blpage{685}
(\byear{1997})
\doiurl{10.2977/prims/1195145150}
\end{barticle}
\endbibitem

\bibitem[\protect\citeauthoryear{Tanab{\'e}}{2004}]{tanabe2004combinatorial}
\begin{botherref}
\oauthor{\bsnm{Tanab{\'e}}, \binits{S.}}:
{Combinatorial aspects of the mixed Hodge structure}.
Preprint at \url{https://arxiv.org/abs/math/0405062}
(2004)
\end{botherref}
\endbibitem

\bibitem[\protect\citeauthoryear{Harder}{2021}]{harder2021hodge}
\begin{barticle}
\bauthor{\bsnm{Harder}, \binits{A.}}:
\batitle{{Hodge numbers of Landau--Ginzburg models}}.
\bjtitle{Advances in mathematics}
\bvolume{378},
\bfpage{107436}
(\byear{2021})
\doiurl{10.1016/j.aim.2020.107436}
\end{barticle}
\endbibitem

\bibitem[\protect\citeauthoryear{Douai}{2018}]{douai2018global}
\begin{barticle}
\bauthor{\bsnm{Douai}, \binits{A.}}:
\batitle{{Global spectra, polytopes and stacky invariants}}.
\bjtitle{Mathematische Zeitschrift}
\bvolume{288}(\bissue{3-4}),
\bfpage{889}--\blpage{913}
(\byear{2018})
\doiurl{10.1007/s00209-017-1918-8}
\end{barticle}
\endbibitem

\bibitem[\protect\citeauthoryear{Douai}{2021}]{douai2021hard}
\begin{barticle}
\bauthor{\bsnm{Douai}, \binits{A.}}:
\batitle{{Hard Lefschetz properties and distribution of spectra in singularity theory and Ehrhart theory}}.
\bjtitle{Journal of Singularities}
\bvolume{23},
\bfpage{116}--\blpage{126}
(\byear{2021})
\doiurl{10.5427/jsing.2021.23g}
\end{barticle}
\endbibitem

\bibitem[\protect\citeauthoryear{Borisov et~al.}{2005}]{borisov2005orbifold}
\begin{barticle}
\bauthor{\bsnm{Borisov}, \binits{L.}},
\bauthor{\bsnm{Chen}, \binits{L.}},
\bauthor{\bsnm{Smith}, \binits{G.}}:
\batitle{{The orbifold Chow ring of toric Deligne-Mumford stacks}}.
\bjtitle{Journal of the American Mathematical Society}
\bvolume{18}(\bissue{1}),
\bfpage{193}--\blpage{215}
(\byear{2005})
\doiurl{10.1090/S0894-0347-04-00471-0}
\end{barticle}
\endbibitem

\bibitem[\protect\citeauthoryear{Sabbah}{2018}]{sabbah2018some}
\begin{barticle}
\bauthor{\bsnm{Sabbah}, \binits{C.}}:
\batitle{{Some properties and applications of Brieskorn lattices}}.
\bjtitle{Journal of Singularities}
\bvolume{18},
\bfpage{238}--\blpage{247}
(\byear{2018})
\doiurl{10.5427/jsing.2018.18k}
\end{barticle}
\endbibitem

\bibitem[\protect\citeauthoryear{Barthel et~al.}{2002}]{barthel2002combinatorial}
\begin{barticle}
\bauthor{\bsnm{Barthel}, \binits{G.}},
\bauthor{\bsnm{Brasselet}, \binits{J.-P.}},
\bauthor{\bsnm{Fieseler}, \binits{K.-H.}},
\bauthor{\bsnm{Kaup}, \binits{L.}}:
\batitle{{Combinatorial intersection cohomology for fans}}.
\bjtitle{Tohoku Mathematical Journal, Second Series}
\bvolume{54}(\bissue{1}),
\bfpage{1}--\blpage{41}
(\byear{2002})
\doiurl{10.2748/tmj/1113247177}
\end{barticle}
\endbibitem

\bibitem[\protect\citeauthoryear{Braden}{2006}]{braden2006remarks}
\begin{barticle}
\bauthor{\bsnm{Braden}, \binits{T.}}:
\batitle{{Remarks on the Combinatorial Intersection Cohomology of Fans}}.
\bjtitle{Pure and Applied Mathematics Quarterly}
\bvolume{2}(\bissue{4}),
\bfpage{1149}--\blpage{1186}
(\byear{2006})
\doiurl{10.4310/PAMQ.2006.v2.n4.a10}
\end{barticle}
\endbibitem

\bibitem[\protect\citeauthoryear{Fleming and Karu}{2010}]{fleming2010hard}
\begin{barticle}
\bauthor{\bsnm{Fleming}, \binits{B.}},
\bauthor{\bsnm{Karu}, \binits{K.}}:
\batitle{{Hard Lefschetz theorem for simple polytopes}}.
\bjtitle{Journal of Algebraic Combinatorics}
\bvolume{32},
\bfpage{227}--\blpage{239}
(\byear{2010})
\doiurl{10.1007/s10801-009-0212-1}
\end{barticle}
\endbibitem

\bibitem[\protect\citeauthoryear{Billera and Rose}{1992}]{billera1992modules}
\begin{barticle}
\bauthor{\bsnm{Billera}, \binits{L.J.}},
\bauthor{\bsnm{Rose}, \binits{L.L.}}:
\batitle{{Modules of piecewise polynomials and their freeness}}.
\bjtitle{Mathematische Zeitschrift}
\bvolume{209}(\bissue{1}),
\bfpage{485}--\blpage{497}
(\byear{1992})
\doiurl{10.1007/BF02570848}
\end{barticle}
\endbibitem

\bibitem[\protect\citeauthoryear{Billera}{1989}]{billera1989algebra}
\begin{barticle}
\bauthor{\bsnm{Billera}, \binits{L.J.}}:
\batitle{{The algebra of continuous piecewise polynomials}}.
\bjtitle{Advances in Mathematics}
\bvolume{76}(\bissue{2}),
\bfpage{170}--\blpage{183}
(\byear{1989})
\doiurl{10.1016/0001-8708(89)90047-9}
\end{barticle}
\endbibitem

\bibitem[\protect\citeauthoryear{Cox et~al.}{2011}]{cox2011toric}
\begin{bbook}
\bauthor{\bsnm{Cox}, \binits{D.A.}},
\bauthor{\bsnm{Little}, \binits{J.B.}},
\bauthor{\bsnm{Schenck}, \binits{H.K.}}:
\bbtitle{Toric Varieties}
vol. \bseriesno{124}.
\bpublisher{American Mathematical Soc.},
\blocation{Providence, Rhode Island}
(\byear{2011})
\end{bbook}
\endbibitem

\bibitem[\protect\citeauthoryear{Brion}{1997}]{brion1997structure}
\begin{barticle}
\bauthor{\bsnm{Brion}, \binits{M.}}:
\batitle{{The structure of the polytope algebra}}.
\bjtitle{Tohoku Mathematical Journal, Second Series}
\bvolume{49}(\bissue{1}),
\bfpage{1}--\blpage{32}
(\byear{1997})
\doiurl{10.2748/tmj/1178225183}
\end{barticle}
\endbibitem

\bibitem[\protect\citeauthoryear{Hertling}{2002}]{hertling2002frobenius}
\begin{bbook}
\bauthor{\bsnm{Hertling}, \binits{C.}}:
\bbtitle{Frobenius Manifolds and Moduli Spaces for Singularities}
vol. \bseriesno{151}.
\bpublisher{Cambridge University Press},
\blocation{Cambridge}
(\byear{2002})
\end{bbook}
\endbibitem

\bibitem[\protect\citeauthoryear{McMullen}{1993}]{mcmullen1993simple}
\begin{barticle}
\bauthor{\bsnm{McMullen}, \binits{P.}}:
\batitle{{On simple polytopes}}.
\bjtitle{Inventiones mathematicae}
\bvolume{113}(\bissue{1}),
\bfpage{419}--\blpage{444}
(\byear{1993})
\doiurl{10.1007/BF01244313}
\end{barticle}
\endbibitem

\bibitem[\protect\citeauthoryear{Gross}{2011}]{gross2011tropical}
\begin{bbook}
\bauthor{\bsnm{Gross}, \binits{M.}}:
\bbtitle{Tropical Geometry and Mirror Symmetry}.
\bpublisher{American Mathematical Soc.},
\blocation{Providence, Rhode Island}
(\byear{2011})
\end{bbook}
\endbibitem

\bibitem[\protect\citeauthoryear{Kouchnirenko}{1976}]{kouchnirenko1976polyedres}
\begin{barticle}
\bauthor{\bsnm{Kouchnirenko}, \binits{A.G.}}:
\batitle{{Polyedres de Newton et nombres de Milnor}}.
\bjtitle{Inventiones mathematicae}
\bvolume{32}(\bissue{1}),
\bfpage{1}--\blpage{31}
(\byear{1976})
\doiurl{10.1007/BF01389769}
\end{barticle}
\endbibitem

\bibitem[\protect\citeauthoryear{Scherk and Steenbrink}{1985}]{scherk1985mixed}
\begin{barticle}
\bauthor{\bsnm{Scherk}, \binits{J.}},
\bauthor{\bsnm{Steenbrink}, \binits{J.H.M.}}:
\batitle{{On the mixed Hodge structure on the cohomology of the Milnor fibre}}.
\bjtitle{Mathematische Annalen}
\bvolume{271},
\bfpage{641}--\blpage{665}
(\byear{1985})
\doiurl{10.1007/BF01456138}
\end{barticle}
\endbibitem

\end{thebibliography}

\end{document}